\documentclass{article}
\usepackage{a4wide}
\usepackage{amsmath}
\usepackage{amssymb}

\usepackage{mathtools}
\usepackage{framed}
\usepackage{amsthm}
\usepackage{fancybox}
\usepackage{graphicx}
\usepackage{color}
\usepackage[noend]{algorithmic}

\usepackage{dsfont}

\def\beq{\begin{equation}}
\def\eeq{\end{equation}}

\def\bea{\begin{eqnarray}}
\def\eea{\end{eqnarray}}

\def\bsp{\begin{split}}
\def\esp{\end{split}}

\def\am{\arg\min}

\def\ol{\overline}
\def\tfrac#1#2{{\textstyle \frac{#1}{#2}}}  

\def\cR{\mathcal{R}}
\def\cS{\mathcal{S}}

\def\N{{{\rm I}\!{\rm N}}}
\def\R{{{\rm I}\!{\rm R}}}

\def\be#1{\begin{equation} \label{#1} }
\def\ee{\end{equation}}
\def\barr{\begin{array}}
\def\earr{\end{array}}

\def\ol#1{\overline{#1}}
\def\tfrac#1#2{{\textstyle \frac{#1}{#2}}}

\def\gdel{g^\delta}

\newcommand{\normklein}[1]{\|{#1}\|}
   
\newcommand{\normGklein}[1]{\normklein{#1}_G}

\newcommand{\qold}{q_{\operatorname{old}}}
\newcommand{\uold}{u_{\operatorname{old}}}
\newcommand{\uoldeins}{u_{\operatorname{old},1}}
\newcommand{\vold}{v_{\operatorname{old}}}

\newcommand{\voldh}{v_{\operatorname{old},h}}
\newcommand{\uoldh}{u_{\operatorname{old},h}}

\newcommand{\qkdelta}{q^{k,\delta}}

\newcommand{\uk}{u^k}
\newcommand{\wk}{w^k}

\newcommand{\qkminuseinsdelta}{q^{k-1,\delta}}
\newcommand{\qh}{q_h}
\newcommand{\xh}{x_h}
\newcommand{\wh}{w_h}
\newcommand{\zh}{z_h}
\newcommand{\uh}{u_h}
\newcommand{\vh}{v_h}
\newcommand{\yh}{y_h}

\newcommand{\xeinsh}{x_h^1}

\newcommand{\xzweih}{x_h^2}
\newcommand{\xdreih}{x_h^3}
\newcommand{\qhzweik}{q^k_{h^2_k}}
\newcommand{\whzweik}{w^k_{h^2_k}}
\newcommand{\qhdreik}{q^k_{h^3_k}}
\newcommand{\whdreik}{w^k_{h^3_k}}
\newcommand{\uoldhdreikplus}{u_{\operatorname{old},h^3_{k+1}}^{k+1}}
\newcommand{\uoldhvierkplus}{u_{\operatorname{old},h^4_{k+1}}^{k+1}}
\newcommand{\whk}{w^k_h}
\newcommand{\uhk}{u^k_h}
\newcommand{\uhdreik}{u^k_{h^3_k}}
\newcommand{\qoldnull}{q_{\operatorname{old}}^0}
\newcommand{\uoldnull}{u_{\operatorname{old}}^0}
\newcommand{\uoldhnullnull}{u_{\operatorname{old},h_0}^0}
\newcommand{\uoldheinsnullnull}{u_{\operatorname{old},h^1_0}^0}
\newcommand{\qhnullnull}{q^0_{h_0}}
\newcommand{\qhnull}{q^0_{h}}

\newcommand{\qhk}{q^k_h}
\newcommand{\qhkk}{q^k_{h_k}}
\newcommand{\uhkk}{u^k_{h_k}}

\newcommand{\whkk}{w^k_{h_k}}

\newcommand{\qk}{q^k}

\newcommand{\qhkminuseinsstern}{q^{k_*-1}_{h_{k_*-1}}}
\newcommand{\qoldksterndeltal}{q_{\operatorname{old}}^{k_*(\delta_l)}}
\newcommand{\qoldkstern}{q_{\operatorname{old}}^{k_*}}
\newcommand{\voldeins}{v_{\operatorname{old}}^1}

\newcommand{\qoldzwei}{q_{\operatorname{old}}^2}
\newcommand{\qolddrei}{q_{\operatorname{old}}^3}
\newcommand{\qoldvier}{q_{\operatorname{old}}^4}
\newcommand{\uoldk}{u_{\operatorname{old}}^k}
\newcommand{\uoldhk}{u_{\operatorname{old},h}^k}
\newcommand{\uoldhkk}{u_{\operatorname{old},h_k}^k}

\newcommand{\qoldk}{q_{\operatorname{old}}^k}
\newcommand{\qkminuseins}{q^{k-1}}

\newcommand{\uoldkplus}{u_{\operatorname{old}}^{k+1}}
\newcommand{\uoldhkplus}{u_{\operatorname{old},h}^{k+1}}
\newcommand{\qoldkplus}{q_{\operatorname{old}}^{k+1}}

\newcommand{\qhkstern}{q_h^{k_*}}

\newcommand{\qhkminuseinsk}{q_{k-1}^{h_{k-1}}}

\newcommand{\Ieinsh}{I_{1,h}}
\newcommand{\Izweih}{I_{2,h}}
\newcommand{\Idreih}{I_{3,h}}
\newcommand{\Ivierh}{I_{4,h}}
\newcommand{\Ieinshk}{I_{1,h}^k}
\newcommand{\Izweihk}{I_{2,h}^k}
\newcommand{\Idreihk}{I_{3,h}^k}
\newcommand{\Ivierhk}{I_{4,h}^k}
\newcommand{\Ieinshkk}{I_{1,h_k}^k}
\newcommand{\Izweihkk}{I_{2,h_k}^k}
\newcommand{\Idreihkk}{I_{3,h_k}^k}
\newcommand{\Ivierhkk}{I_{4,h_k}^k}

\newcommand{\Ieinsk}{I_{1}^k}
\newcommand{\Izweik}{I_{2}^k}
\newcommand{\Idreik}{I_{3}^k}
\newcommand{\Ivierk}{I_{4}^k}
\newcommand{\uttheta}{\tilde{\underline\theta}}
\newcommand{\ottheta}{\tilde{\overline \theta}}

\newcommand{\dualW}[2]{\langle#1,#2\rangle_{W^*,W}}    

\author{B.~Kaltenbacher \and A.~Kirchner \and S.~Veljovi\' c}
\title{Goal oriented adaptivity in the IRGNM for parameter identification in PDEs I:\\ 
reduced formulation}

\newtheorem{thm}{Theorem}
\newtheorem{prop}{Proposition}

\newtheorem{rem}{Remark}

\newtheorem{ass}{Assumption}
\newtheorem{algorithm}{Algorithm}{\bf}{\it}

\begin{document}

\maketitle

\abstract{In this paper we study adaptive discretization of the iteratively regularized Gauss-Newton method IRGNM 
  with an a posteriori (discrepancy principle) choice of the regularization parameter in each Newton step and of 
  the stopping index. We first of all prove convergence and convergence rates under some accuracy requirements 
  formulated in terms of four quantities of interest. Then computation of error estimators for these quantities 
    based on a weighted dual residual method is discussed, which results in an algorithm for adaptive refinement. 
    Finally we extend the results from the Hilbert space setting with quadratic penalty to Banach spaces and general 
    Tikhonov functionals for the regularization of each Newton step.}

\section{Introduction}

Parameter identification problems in partial differential equations (PDEs) can often be written as nonlinear
ill-posed operator equations
\beq\label{OEq}
F(q) = g,
\eeq
where $F$ is a nonlinear operator between Hilbert spaces $Q$ and $G$ and where
the given data $g^\delta$ is noisy with the noise level $\delta$:
\beq\label{noise}
||g - g^\delta|| \leq \delta.
\eeq
Throughout this paper we will assume that a solution $q^\dagger$ to \eqref{OEq} exists.

In case of inverse problems for PDEs, $F$ is the  composition of a parameter-to-solution map 
$$\begin{array}{rcl}
S\colon Q&\to&V\\
q&\mapsto&u
\end{array}$$ 
with some measurement operator 
$$\begin{array}{rcl}
C\colon V&\to&G\\
u&\mapsto&g\,,
\end{array}$$
where $V$ is an appropriate Hilbert space.
Here, we will write the underlying (possibly nonlinear) PDE in its weak form:
\beq\label{PDEintro} 
{ \text{For } q\in Q \text{ find } u\in V:}\quad A(q,u)(v)=(f,v) \quad \forall v\in W \,,
\eeq
where $u$ denotes the PDE solution, $q$ some searched for coefficient or boundary function, and $f\in W^*$ is
some given right hand side in the dual of some Hilbert space $W$.
We will assume that the PDE \eqref{PDEintro} and especially also its linearization at $(q,u)$ is uniquely and
stably solvable.


For the stable solution of \eqref{OEq} with noisy data, 
we consider the iteratively regularized Gauss-Newton method (IRGNM)
first of all (section \ref{sec_redIRGNM}) in the reduced form 
\beq \label{redIRGNM}
\qkdelta = \qkminuseinsdelta - (F'(\qkminuseinsdelta)^* F'(\qkminuseinsdelta) + \alpha_{k}
I)^{-1}(F'(\qkminuseinsdelta)^* (F(\qkminuseinsdelta)-g^\delta) + \alpha_k(\qkminuseinsdelta - q_0))
\eeq
or equivalently
\beq\label{defqeq}
\qkdelta \in \am_q ({\cal T}_{\alpha_k}(q))
=\am_q \|F'(\qkminuseinsdelta)(q-\qkminuseinsdelta) +F(\qkminuseinsdelta)-g^\delta\|_G^2
+ \alpha_k \|q-q_0\|_Q^2,
\eeq
(see, e.g., \cite{BakuKokurin,KNSBuch} and the references therein).
For some all-at-once formulations of the IRGNM we refer to \cite{KKV13} {(part II of this paper)}. 

The regularization parameter $\alpha_k$ and the overall stopping index $k_*$ have to be chosen in an
appropriate way in order to guarantee convergence. We will here use an inexact Newton / discrepancy principle
type strategy, as it has been shown to yield convergence of the IRGNM even in a Banach space setting in
\cite{KSS09}, see also \cite{HohageWerner} for a convergence analysis in a still more general setup but with
different parameter choice strategies for $\alpha_k$ and $k_*$.

Our aim is to consider adaptively discretized versions of the formulations \eqref{redIRGNM} defined by replacing the spaces $Q$, $V$, $W$ with
finite dimensional counterparts $Q_h$, $V_h$, $W_h$ (using possibly different discretizations of $V,W$ in
\eqref{PDElin} and \eqref{PDE}).
These should be sufficiently precise so that the convergence results from the continuous setting can be
carried over, but save computational effort by using degrees of freedom only where really necessary.
For this purpose we will make use of goal oriented error estimators {(\cite{BeckerRannacher,BeckerVexler})}, 
that control the error in some quantities of interest $I$, which are functionals of the variables $q,u,w$ (see \eqref{eq_IRGNM_optprob}-\eqref{PDElin} below). 
{ We follow the concept proposed in \cite{GKV}, where an inexact Newton method for the computation of a regularization parameter according to
the discrepancy principle is combined with adaptive refinement using goal oriented error estimators. 
While \cite{GKV} is limited to linear inverse problems, in \cite{KKV10} the idea has been extended to the nonlinear case.
Different from \cite{KKV10}, we do not treat the nonlinear problem directly here, but use an iterative solution algorithm, the iteratively 
regularized Gauss-Newton
method \eqref{redIRGNM}, \eqref{defqeq} and treat a sequence of linearized problems instead.}

The remainder of this paper is organized as follows. In Section \ref{sec_redIRGNM} we formulate the Newton step equation as linear quadratic optimal control problem and derive its discretization together with certain quantities of interest, whose precision will be crucial for obtaining convergence results for the overall regularized Newton iteration. This will be substantiated in the 
convergence and convergence rates results provided in the subsections \ref{subsec_conv_redIRGN} and \ref{subsec_rates_redIRGN}. Subsection 
\ref{subsec_estimatorsIRGNM} describes computation of the required error estimators by a goal oriented approach and Subsection \ref{subsec_algorithm} provides the full algorithm. The method and its analysis is extended to a setting with general data misfit and regularization terms in Subsection \ref{subsec_generalize}. We conclude with a few remarks in Section \ref{sec_conclusions}.

\section{Reduced form of the discretized IRGNM} \label{sec_redIRGNM}

We consider the iteration rule \eqref{redIRGNM} for solving the optimization problem  
\beq\label{eq_IRGNM_defqeq}
\min_{q\in Q} \|F'(q^{\delta,k-1})(q-q^{\delta,k-1}) +F(q^{\delta,k-1})-g^\delta\|_G^2
+ {\frac 1 {\beta_k}} \|q-q_0\|_Q^2\,,
\eeq 
{where the regularization parameter $\beta_k$ is updated in each Gauss Newton iteration according to an inexact Newton method guaranteeing 
a relaxed version of the discrepancy principle (see Step 15 in Algorithm \ref{alg:GKVAlgorithm}, \cite{GKV,KKV10}).} 
Note that although the domain $\mathcal{D}(F)$ might be a strict subset of $Q$, we need not explicitly restrict $q$ to $\mathcal{D}(F)$ in this minimization, 
since we will assume that $\mathcal{D}(F)$ contains a ball of radius $\rho$ around $q_0$ and prove that all iterates remain in this ball, cf. \eqref{ineq0}. 
So minimizers over $\mathcal{D}(F)$ will automatically be minimizers over $Q$.

We start with a detailed description of a single iteration step \eqref{redIRGNM} for fixed (discretized)
previous
iterate $q^{\delta,k-1}=\qold \in Q$ in a continuous and later in the discretized setting actually used in
computations, along with the quantities of interest required in error estimation and adaptive refinement.

We formulate the optimization problem \eqref{eq_IRGNM_defqeq} 
as optimal control problem 
\be{eq_IRGNM_optprob}
    \min_{(q,\uold,w)\in Q \times V \times V}  \|C'(\uold)(w)+C(\uold) -g^\delta\|_G^2
	+ {\frac 1 {\beta_k}} \|q-q_0\|_Q^2
\ee
\begin{align}
  \text{s.t.} \qquad\qquad A(\qold,\uold)(v) &= f(v)& \forall v\in W\,,\label{PDE}\\
  A_u'(\qold,\uold)(w,v) 
 &= - A_q'(\qold,\uold)(q-\qold,v) & \forall v\in W\label{PDElin}\,,
\end{align}
since { for a solution $q^{\delta,k}$ of \eqref{eq_IRGNM_defqeq}} $(q^{\delta,k},S(\qold),S'(\qold)(q^{\delta,k} - \qold))$ solves \eqref{eq_IRGNM_optprob}-\eqref{PDElin}.

In most of this section we omit the superscript $\delta$ (denoting dependence on the noisy data) 
in order to be able to better indicate the difference between continuous and discretized 
quantities. 

We consider the following quantities of interest 
\beq \label{eq_IRGNM_I1234}
\begin{aligned}
\tilde I_1\colon &\ Q \times V\times V\times \R \to \R\,, &&(q,\uold,w,\beta) &&\mapsto \|C'(\uold)(w)+C(\uold)-g^\delta\|_G^2+{\frac  1 \beta}\|q-q_0\|_Q^2\\
\tilde I_2\colon &\ V \times V\to \R\,, && (\uold,w) &&\mapsto \|C'(\uold)(w)+C(\uold) -g^\delta\|_G^2\\
\tilde I_3\colon &\ V \to \R\,,  &&\uold&&\mapsto \|C(\uold) -g^\delta\|_G^2\,,
\end{aligned}
\eeq
i.e. we assume the knowledge about error estimates
\begin{align*}
\eta_1 &\ge |\tilde I_1(q,\uold,w,\beta) - \tilde I_1(q_h,\uoldh,w_h,\beta)|\\
\eta_2 &\ge |\tilde I_2(\uold,w) - \tilde I_2({\uold}_h,w_h)|\\
\eta_3 &\ge |\tilde I_3(\uold) - \tilde I_3({\uold}_h)|\,,
\end{align*}
where $q_h,\uoldh,w_h$ is a discrete approximate solution to \eqref{eq_IRGNM_optprob}, which will 
be concretised in the following. {The error bounds $\eta_1$, $\eta_2$ and $\eta_3$ will be estimated 
 using goal oriented error estimators cf. Section \ref{subsec_estimatorsIRGNM}.}

Additionally we define the functionals  
\beq \label{eq_IRGNM_I1234reduced}
\begin{aligned}
I_1\colon &\ Q\times Q\times \R \to  \R\,, &&(\qold,q,\beta)&&\mapsto \|F'(\qold)(q-\qold)+F(\qold)-g^\delta\|_G^2+{\frac  1 \beta}\|q-q_0\|_Q^2\\
I_2\colon &\ Q\times Q \to \R\,, &&(\qold,q)&&\mapsto \|F'(\qold)(q-\qold)+F( \qold) -g^\delta\|_G^2\\
I_3\colon &\ Q\to \R\,, &&\qold&&\mapsto \|F( \qold) -g^\delta\|_G^2\\
I_4\colon &\ Q\to \R\,, &&q&&\mapsto \|F(q) -g^\delta\|_G^2\,,
\end{aligned}
\eeq
which can be seen as reduced versions of \eqref{eq_IRGNM_I1234}, since for a solution $(q,\uold,w)$ of 
\eqref{eq_IRGNM_optprob}-\eqref{PDElin} and $u \in V$ fulfilling 
\be{eq_IRGNM_currentstateequation}
	A(q,u)(v) = f(v) \qquad \forall v\in W
\ee
there holds 
\be{eq_IRGNM_ItildeI}
\begin{aligned}
	\tilde I_1(q,\uold,w,\beta) &=I_1(\qold,q,\beta)\,,&
	\quad \tilde I_2(\uold,w) &= I_2(\qold,q)\,,\\
	\quad \tilde I_3(\uold) &= I_3(\qold)\,,&
	\quad \tilde I_3(u) &= I_4(q)\,.
\end{aligned}
\ee 

The (continuous) quantities of interest in the $k$-th iteration step are then defined as follows: 
For a solution $(\qk,\uoldk,\wk)$ of \eqref{eq_IRGNM_optprob} for given $\qold = \qoldk$ and $\beta = \beta_k$
and $\uk$ fulfilling 
\be{eq_IRGNM_currentstateequationk}
	A(q^k,u^k)(v) = f(v) \qquad \forall v\in W
\ee
in the $k$-th iteration let
\beq \label{eq_IRGNM_I1234k}
\begin{aligned}
I_1^k & \coloneqq  \tilde I_1(\qk, \uoldk, \wk, \beta_k)  
      	= I_1(\qk,\qoldk, \beta_k) \\
      & = \|F'(\qoldk)(\qk-\qoldk)+F(\qoldk)-g^\delta\|_G^2+{\frac 1 {\beta_k}}\|\qk-q_0\|_Q^2\\
I_2^k & \coloneqq  \tilde I_2(\qoldk,\wk)  
      	= I_2(\qoldk,\qk)
       = \|F'(\qoldk)(\qk-\qoldk)+F(\qoldk)-g^\delta\|_G^2\\
I_3^k &\coloneqq  \tilde I_3(\uoldk)
       = I_3(\qoldk)
      	= \|F(\qoldk) -g^\delta\|_G^2\\
I_4^k & \coloneqq  \tilde I_3(\uk) 
	= I_4(\qk)
       = \|F(\qk) -g^\delta\|_G^2\,.
\end{aligned}
\eeq
{ To formulate the quantities of interest \eqref{eq_IRGNM_I1234} for a discrete setting,}
we consider finite element spaces $Q_h,V_h,W_h$ to $Q,V,W$, and $S_h$ denotes the discrete solution operator of the 
state equation. The discretized version of the optimal control problem \eqref{eq_IRGNM_optprob}
for given $\qold\in Q_h$ can then be formulated as 
\beq\label{eq_IRGNM_minJ1Discrete} 
\min_{(q,\uold,w)\in Q_h \times V_h \times V_h} 
\|C'(\uold)(w)+C(\uold)-g^\delta\|_G^2+{\frac  1 \beta}\|q-q_0\|_Q^2
\eeq
subject to 
\begin{align}  
A(\qold,\uold)(v)&=f(v)& \forall v\in W_h
\label{PDEh}\\
A'_u(\qold,\uold)(w,v)+A'_q(\qold,\uold)(q-\qold,v)&=0 & \forall v \in W_h\,. 
\label{PDElinh}
\end{align}

Equation \eqref{PDEh} is equivalent to
$\uold = S_h(\qold)$ and \eqref{PDElinh} is equivalent to 
$w = S_h'(\qold)(q-\qold)$, such that the reduced form of \eqref{eq_IRGNM_minJ1Discrete} reads
\beq\label{eq_IRGNM_minJ1Discretereduced} 
\min_{q \in Q_h} 
\|F_h'(\qold)(q-\qold)+F_h(\qold)-g^\delta\|_G^2+{\frac  1 \beta}\|q-q_0\|_Q^2
\eeq
with $F_h = C \circ S_h$.

\begin{rem}
One can think of using different discretizations $(V_h,W_h)$ for \eqref{PDEh} and
$(\tilde{V}_h,\tilde{W}_h)$ for \eqref{PDElinh} (see Algorithm \ref{alg_IRGNM}), which we do not indicate here in order to avoid a
too complicated setup that would probably not lead to much gain in computational efficiency.
\end{rem}

Then the discrete quantities of interest in the reduced form (i.e. the discrete counterparts to 
\eqref{eq_IRGNM_I1234reduced}) are defined by
\beq \label{eq_IRGNM_I1234reduceddiscrete}
\begin{aligned}
\Ieinsh\colon &\ Q \times Q \times \R \to  \R\,, &&(\qold,q,\beta)&&\mapsto \|F_h'(\qold)(q-\qold)+F_h(\qold)-g^\delta\|_G^2+{\frac 1 \beta}\|q-q_0\|_Q^2\\
\Izweih\colon &\ Q\times Q \to \R\,, &&(\qold,q)&&\mapsto \|F_h'(\qold)(q-\qold)+F_h( \qold) -g^\delta\|_G^2\\
\Idreih\colon &\ Q\to \R\,, &&\qold&&\mapsto \|F_h( \qold) -g^\delta\|_G^2\\
\Ivierh\colon &\ Q\to \R\,, &&q&&\mapsto \|F_h(q) -g^\delta\|_G^2\,,
\end{aligned}
\eeq
such that consistent with \eqref{eq_IRGNM_ItildeI} for a solution $(q_h,\uoldh,w_h)$ of 
the discretized problem \eqref{eq_IRGNM_minJ1Discrete}-\eqref{PDElinh} there holds
\begin{align*}
	\tilde I_1(q_h,\uoldh,w_h,\beta_h) &= \Ieinsh(\qold,q_h,\beta)\,,
	&\quad \tilde I_2(\uoldh,w_h) &= \Izweih(\qold,q_h)\,,\\
	\quad \tilde I_3(\uoldh) &= \Idreih(\qold)\,,
	&\quad \tilde I_3(u_h) &= \Ivierh(q_h)\,.
\end{align*}

Correspondingly, the discrete quantities of interest in the $k$-th iteration step 
(i.e. the discrete counterparts to 
\eqref{eq_IRGNM_I1234k}) for a solution $(\qhk,\uoldhk,\whk)$ of \eqref{eq_IRGNM_minJ1Discrete}
for given $\qold = \qoldk \in Q_h$ can be formulated as
\beq \label{eq_IRGNM_I1234kdiscrete}
\begin{aligned}
\Ieinshk & \coloneqq  \tilde I_1(\qhkk , \uoldhkk , \whkk, \beta^k_{h_k} )  
       	=  \Ieinshkk (\qhkk ,\qoldk, \beta_k) \\
       & = \|F_{h_k} '(\qoldk)(\qhkk -\qoldk)+F_{h_k} (\qoldk)-g^\delta\|_G^2+{\frac 1 {\beta_k}}\|\qhkk -q_0\|_Q^2\\
 \Izweihk & \coloneqq  \tilde I_2(\qoldk,\whkk)  
       	= \Izweihkk(\qoldk,\qhkk)
        = \|F_{h_k} '(\qoldk)(\qhkk -\qoldk)+F_{h_k} (\qoldk)-g^\delta\|_G^2\\
 \Idreihk &\coloneqq  \tilde I_3(\uoldhkk)
        = \Idreihkk(\qoldk)
       	= \|F_{h_k} (\qoldk) -g^\delta\|_G^2\\
 \Ivierhk & \coloneqq  \tilde I_3(\uhkk)
        = \Ivierhkk(\qhkk)
        = \|F_{h_k} (\qhkk) -g^\delta\|_G^2\,,
 \end{aligned}
\eeq
where we introduced the notation $h_k$ (replacing $h$), denoting the discretization in step $k$,  
in order to distiguish between the possibly different discretizations during the iterative process in the following.

Note that the norms in $G$ and in $Q$ (and later on also the one in $V$)  
{
as well as the operator $C$ and the semilinear form $a\colon Q\times V\times W \to \R$ defined by the relation 
$a(q,u)(v) = \dualW{A(q,u)}{v}$} (where $\dualW{.}{.}$ denotes the duality pairing between $W^*$ and $W$)
are assumed to be evaluated exactly.
 
At the end of each iteration step we set 
 \begin{equation}\label{eq:qoldkplus}
	\qoldkplus \coloneqq  \qhk\,.
\end{equation}

\begin{rem}\label{rem:auxit}
The sequence of iterates we actually consider is the discrete one $(\qhkk)_{k\in\N}$, which we also update
according to \eqref{eq:qoldkplus}. Besides that, for theoretical purposes we keep a sequence of continuous
iterates $(\qk)_{k\in\N}$, where each member $\qk$ of this sequence emerges from a member
$\qoldk=\qhkminuseinsk$ of the sequence of discretized iterates $(\qhkk)_{k\in\N}$, but {\em not } from
$\qkminuseins$, see Figure \ref{fig:iteration}.
\begin{figure}
\begin{center}
\setlength{\unitlength}{2cm}
\begin{picture}(5,4)
\thinlines
\put(0,0.5){\vector(1,0){4.5}}
\put(-0.05,0.45){$\bullet$}
\put(0.95,0.45){$\bullet$}
\put(1.95,0.45){$\bullet$}
\put(2.95,0.45){$\bullet$}
\put(3.95,0.45){$\bullet$}
\put(-0.2,0.3){$k=0$}
\put(0.8,0.3){$k=1$}
\put(1.8,0.3){$k=2$}
\put(2.8,0.3){$k=3$}
\put(3.8,0.3){$k=4$}
\thicklines
\put(0,1){\line(2,1){1}}
\put(1,1.5){\line(4,3){1}}
\put(2,2.25){\line(1,1){1}}
\put(3,3.25){\line(4,3){1}}
\put(-0.05,0.95){$\bullet$}
\put(0.95,1.45){$\bullet$}
\put(1.95,2.2){$\bullet$}
\put(2.95,3.2){$\bullet$}
\put(3.95,3.95){$\bullet$}
\put(0.05,0.85){$q_0$}
\put(1.05,1.35){$q^1_{h_1}=\qoldzwei$}
\put(2.05,2.1){$q^2_{h_2}=\qolddrei$}
\put(3.05,3.1){$q^3_{h_3}=\qoldvier$}
\put(4.05,3.85){$q^4_{h_4}$}
\thinlines
\put(0,1){\line(3,2){1}}
\put(1,1.5){\line(2,1){1}}
\put(2,2.25){\line(2,1){1}}
\put(3,3.25){\line(2,1){1}}
\put(0.95,1.61){$\diamond$}
\put(1.95,1.95){$\diamond$}
\put(2.95,2.7){$\diamond$}
\put(3.95,3.7){$\diamond$}
\put(1.05,1.75){$q_1$}
\put(2.05,1.9){$q_2$}
\put(3.05,2.6){$q_3$}
\put(4.05,3.6){$q_4$}

\end{picture}
\end{center}
\caption{Sequence of discretized iterates and auxiliary sequence of continuous iterates}\label{fig:iteration}
\end{figure}
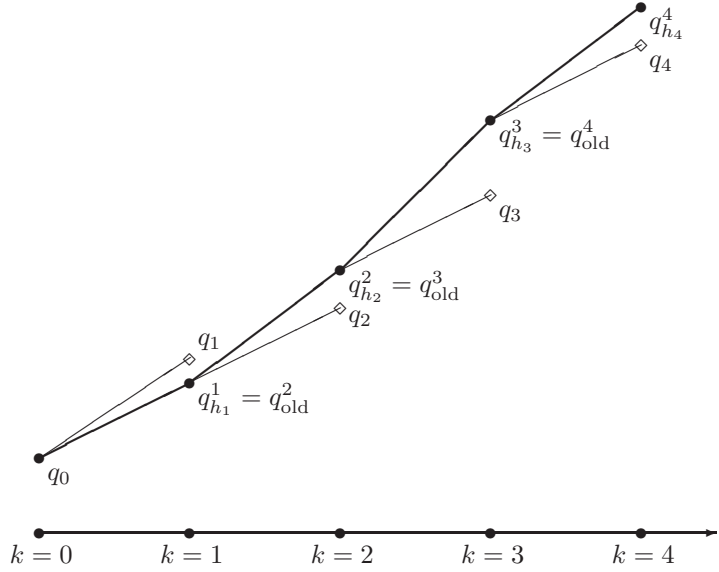
One of the reasons for the necessity of considering this auxiliary continuous iterates is the key inequality
\eqref{min} in the proof of the convergence theorem below, which makes use of minimality of the iterate $q_k$
in all of $Q$ (and not only in the finite dimensional subspace $Q_h$) thus allowing for comparison to the infinite dimensional exact solution $q^\dagger$. 

We stress once more that the discretization may be different in each iteration, as  indicated by the
superscripts $h_k$, $h_{k-1}$ here. In order to keep the notation readable we will suppress the iteration
index $k$ in the superscript $h_k$ whenever this is possible without causing confusion. 
\end{rem}

\begin{rem}\label{remI3I4}
In view of \eqref{eq:qoldkplus} 
and the last two identities in \eqref{eq_IRGNM_ItildeI} 
one might think that $\Idreih^{k+1} = \|F(\qoldkplus) -g^\delta\|_G^2$ and 
$\Ivierhk = \|F(\qhk) -g^\delta\|_G^2$ are the same, but this is not the case, since
there holds indeed 
\[
	\uhk = S_{h_k}(\qhk) = S_{h_k}(\qoldkplus) 
\]
for $h=h_k$, i.e. with respect to the discretization from step $k$, but
\[
	\uoldhkplus = S_{h_{k+1}}(\qoldkplus) 
\]
for $h=h_{k+1}$, i.e. with respect to the discretization from step $k+1$. Due to the possibly 
different discretizations, in general there holds 
\[
	\uhk \neq \uoldhkplus\,.
\]
Also $I_3^{k+1}= \|F(\qoldkplus) -g^\delta\|_G^2$  and $I_4^k = \|F(\qk) -g^\delta\|_G^2$ are 
not the same, because 
\[
	\qoldkplus\coloneqq  \qhk \neq \qk\,,
\]
see Figure \ref{fig:iteration}.
\end{rem}

\begin{rem}
Note that even the discretizations $h_k$ for fixed $k$ can differ in the different quantities of interest
{during one Gauss Newton iteration cf. Algorithm \ref{alg_IRGNM}.} 
Tracking the proof of the main convergence result Theorem \ref{conv} the reader can verify
that only  $I_{1,h}^k$ and $I_{2,h}^k$  have to be evaluated on the same mesh, since in the proof we will need 
the identity 
 \beq\label{I1I2}
 I_{1,h}^k= I_{2,h}^k + {\frac 1 {\beta_k}}\|q^k_{h_k} -q_0 \|^2\,,
 \eeq 
 which is guaranteed by assuming exact evaluation of the $Q$-norm $\|q^k_h -q_0 \|_Q$.
\end{rem}

In order to assess and -- by adaptive refinement -- to control the differences 
\beq\label{intcond}
|I_{i,h}^{k} - I_i^{k}| \leq \eta_i^{k} \,, \quad i\in\{1,2,3,4\}
\eeq
between the exact quantities of interest and their counterparts resulting from discretization, 
we will make use of goal oriented error estimators, which will be explained in more detail in 
Section \ref{subsec_estimatorsIRGNM}.

We select $\beta_k$ according to an inexact Newton condition (cf. \cite{Hanke,Rieder}) which can 
be interpreted as a discrepancy principle with ``noise level'' $\tilde\theta I_{3,h}^{k}$
\beq\label{inexNewton}
{\tilde{\underline\theta}} I_{3,h}^{k} \leq 
I_{2,h}^{k}
\leq {\tilde{\overline\theta}} I_{3,h}^{k} \,,
\eeq
i.e.,
$$
{\tilde{\underline\theta}} \|F_h(\qoldk) -g^\delta\|_G^2 \leq 
\|F_h'(\qoldk)(\qhk-\qoldk)+F_h(\qoldk)-g^\delta\|_G^2
\leq {\tilde{\overline\theta}} \|F_h(\qoldk) -g^\delta\|_G^2 \,,
$$
for some $0 < \tilde{\underline\theta} \leq \tilde{\overline\theta} <\frac 12$. Note that this regularization
parameter can be computed in an efficient manner according to \cite{GKV}, see Theorem 1 there. 
We mention in passing that the latter
would as well allow us to use the continuous version $I_2^{k}$ in \eqref{inexNewton}, but we prefer to formulate the condition with the discretized actually computed quantities anyway.

The overall Newton iteration is stopped according to a generalized discrepancy principle
\beq\label{stop}
k_* = \min \{ k\in \mathbb{N} ~ : ~ I_{3,h}^{k} \leq \tau^2\delta^2 \}.
\eeq

In our convergence analysis we will use the following weak sequential closedness assumption on $F$:
\beq\label{Fseqclosed}
(q_n \rightharpoonup q \land F(q_n) \to g) \Rightarrow (q \in \mathcal{D}(F) \land F(q) = g)
\eeq
for all $\{q_n\}_{n \in \mathbb{N}} \subseteq Q$ together with the tangential cone condition (also often called Scherzer condition)
\beq\label{tangcone}
\|F(q) - F(\bar{q}) - F'(q)(q - \bar{q})\|_G \leq c_{tc} \| F(q) - F(\bar{q})\|_G \qquad \forall q, \bar{q}
\in \mathcal{B}_{\rho}(q_0) \subseteq \mathcal{D}(F)\subseteq Q 
\eeq
for some $\rho>0$,  $0 < c_{tc} < 1$, which are both typical conditions in the analysis of regularization methods for nonlinear ill-posed problems cf., e.g., \cite{EHNBuch,KNSBuch} and the references therein.

\subsection{Convergence} \label{subsec_conv_redIRGN}
\begin{thm}\label{conv}
Let $F$ satisfy the weak sequential closedness condition \eqref{Fseqclosed} and the tangential cone condition \eqref{tangcone}
with $c_{tc} < \frac 1 4$ sufficiently small.
Let, further, $\tau>0$ be chosen sufficiently large and $0<\tilde{\underline\theta}<\tilde{\overline\theta}$ sufficiently small , such that
\beq\label{taucond}
2\left(c_{tc}^2 + \frac{(1+c_{tc})^2}{\tau^2}\right) < \tilde{\underline\theta}
\quad \textrm{and}\quad 
\frac{2\tilde{\overline\theta} + 4c_{tc}^2}{1-4c_{tc}^2} < 1\,.
\eeq
Finally, let for the discretization error with respect to the quantities of interest 
\eqref{intcond} hold,
where $\eta_1^k$, $\eta_2^k$, $\eta_3^k$, $\eta_4^k$ are selected such that
\beq\label{etacond1}
\eta_1^{k} + 2 c_{tc}^2 \eta_3^{k}\leq
\left({\tilde{\underline\theta}} - 2\left(c_{tc}^2 + \frac{(1+ c_{tc})^2}{\tau^2}\right)\right)
I_{3,h}^{k}
\eeq
\beq\label{etacond2}
\eta_3^{k}\leq c_1 I_{3,h}^{k} \ \mbox{ and } \
\eta_2^{k}\to0\,, \ \eta_3^{k}\to0\,, \ \eta_4^{k}\to0 \mbox{ as }k\to\infty 
\eeq
\beq\label{etacond3}
I_{3,h}^{k}\leq (1+c_3)I_{4,h}^{k-1}+r^k \ \mbox{ and } \
(1+c_3)\frac{2{\tilde{\overline\theta}} + 4 c_{tc}^2}{1-4c_{tc}^2}\leq c_2<1
\eeq
for some constants $c_1, c_2, c_3 >0$, and a sequence $r^k\to0$ as $k\to\infty$ (where the second condition in \eqref{etacond3} 
is possible due to the right inequality in \eqref{taucond}).

Then with $\beta_k$  and $h=h_k$ fulfilling \eqref{inexNewton}
and $k_*$ selected according to \eqref{stop} there holds 
\begin{itemize}
\item[(i)] For any solution $q^\dag \in \mathcal{B}_{\rho}(q_0)$ of \eqref{OEq}
  \beq\label{ineq0}
  \|\qhk -q_0\|^2_Q \leq \|q^\dag -q_0\|^2_Q \quad \forall k<k_*\,,
  \eeq
\item[(ii)] $k_*$ is finite,
\item[(iii)] $\qoldkstern = \qhkminuseinsstern$ converges (weakly) subsequentially to a solution of \eqref{OEq} as $\delta\to0$
in the sense that it has a weakly convergent subsequence and each
weakly convergent subsequence converges strongly to a solution of \eqref{OEq}.
If the solution $q^\dagger$ to \eqref{OEq} is unique, then $\qoldkstern$ converges strongly to
$q^\dagger$ as $\delta\to0$.
\end{itemize}
\end{thm}
\begin{proof}
\begin{itemize}
\item[(i):]
For $k=0$, \eqref{ineq0} trivially holds.
For all $1\leq k<k_*$ 
and any solution $q^\dag$ of \eqref{OEq} we have by \eqref{intcond} and minimality of $\qk$
\begin{equation}\label{min}
I_{1,h}^{k} \leq I_1^{k} + \eta_1^{k} = 
\| F'(\qoldk)(q^\dag-\qoldk) +F(\qoldk) -g^\delta\|_G^2 + {\frac 1 {\beta_k}} \|q^\dag -q_0\|_Q^2 + \eta_1^{k}\,.
\end{equation}
In here, according to \eqref{noise},  \eqref{stop} and \eqref{tangcone},
as well as the inequaltity $(a+b)^2 \le 2a^2 + 2b^2$ for arbitrary $a,b\in \R$ we can estimate as follows
\begin{align}
\lefteqn{\| F'(\qoldk)(q^\dag-\qoldk) + F(\qoldk) -g^\delta \|_G^2}\nonumber\\
&\leq \left(\| F'(\qoldk)(q^\dag-\qoldk) + F(\qoldk) - F(q^\dag) \|_G+\delta\right)^2
\nonumber\\
&\leq \left(c_{tc}\|F(q^\dag)- F(\qoldk) \|_G+\delta\right)^2\nonumber\\
&\leq \left(c_{tc}\left(\|g^\delta- F(\qoldk) \|_G+\delta\right)+\delta\right)^2\nonumber\\
& =  \left(c_{tc}\sqrt{I_3^k} + \left(1+c_{tc}\right)\delta\right)^2\nonumber\\
& \le  2c_{tc}^2 I_3^k+ 2(1+c_{tc})^2 \delta^2\nonumber\\
&\leq 2c_{tc}^2(I_{3,h}^{k} + \eta_3^{k})+ 2(1+c_{tc})^2\frac{I_{3,h}^{k}}{\tau^2}\nonumber\\
&=  2\left(c_{tc}^2+\frac{(1+c_{tc})^2}{\tau^2}\right) I_{3,h}^{k}
+2c_{tc}^2 \eta_3^{k} \label{estTaylor}\,.
\end{align}
On the other hand, from \eqref{I1I2}, \eqref{inexNewton} it follows that
\begin{equation}\label{min1}
I_{1,h}^{k}= I_{2,h}^{k} + {\frac 1 {\beta_k}}\|\qhk -q_0 \|_Q^2
\geq {\tilde{\underline\theta}} I_{3,h}^{k}  + {\frac 1 {\beta_k}} \|\qhk -q_0\|_Q^2
\end{equation}
which together with the previous inequality and \eqref{etacond1} gives
\begin{align*}
\tilde{\underline\theta} I_{3,h}^{k} +  {\frac 1 {\beta_k}} \|\qhk -q_0\|_Q^2 
&\leq I_1^k + \eta_1^k \\
&= \| F'(\qoldk)(q^\dag-\qoldk) + F(\qoldk) -g^\delta \|_G^2 + {\frac 1 {\beta_k}} \|q^\dagger-q_0\|_Q^2 + \eta_1^k\\
&\leq 2\left(c_{tc}^2+\frac{(1+c_{tc})^2}{\tau^2}\right) I_{3,h}^{k}
 + {\frac 1 {\beta_k}} \|q^\dag -q_0\|_Q^2 + \eta_1^{k} 
+ 2c_{tc}^2 \eta_3^{k}\\
& \leq {\tilde{\underline\theta}} I_{3,h}^{k} + {\frac 1 {\beta_k}} \|q^\dag -q_0\|_Q^2\,,
\end{align*}
which implies \eqref{ineq0}.
\item[(ii):] Furthermore, for all  
$1\leq k< k_*$ 
we have by the triangle inequality as well as \eqref{tangcone} and \eqref{inexNewton}
\begin{align}
\sqrt{I_4^{k}} & = \|F(\qk) - g^\delta\|_G\nonumber\\
& \leq \|F'(\qoldk)(\qk - \qoldk) +F(\qoldk)-g^\delta\|_G
+\|F'(\qoldk)(\qk - \qoldk) -F(\qk)+ F(\qoldk)\|_G\nonumber\\
& \leq \sqrt{I_2^{k}} + c_{tc} \|F(\qk) - F(\qoldk)\|_G\nonumber\\
& \leq \sqrt{{\tilde{\overline\theta}} I_{3,h}^{k} +\eta_2^{k}} 
+ c_{tc} (\sqrt{I_4^{k}}+\sqrt{I_3^{k}})\,, 
\label{estI4_1}
\end{align}
hence by $(a+b)^2\leq 2a^2+2b^2$ for all $a,b \in \R$
\[
I_4^{k}\leq 2 ({\tilde{\overline\theta}} I_{3,h}^{k}+\eta_2^{k}) + 2 c_{tc}^2 (2I_4^{k}+2I_3^{k})\,,
\]
which implies
\[
  I_4^k \le \frac 1 {1-4c_{tc}^2}\left(2 {\tilde{\overline\theta}} I_{3,h}^{k} + 2 \eta_2^k + 4c_{tc}^2 I_3^k \right)\,.
\]
With \eqref{intcond} and \eqref{etacond3} we can further deduce
\begin{align*}
I_{4,h}^{k} &\le \frac 1 {1-4c_{tc}^2}\left((2 {\tilde{\overline\theta}} + 4c_{tc}^2)I_{3,h}^{k} + 2 \eta_2^k + 4c_{tc}^2 \eta_3^k \right) + \eta_4^k\\
&\le \frac{2 {\tilde{\overline\theta}} + 4c_{tc}^2}{1-4c_{tc}^2} (1+c_3)I_{4,h}^{k-1}  
+ \frac 1 {1-4c_{tc}^2} \left((2 {\tilde{\overline\theta}} + 4c_{tc}^2)r^k + 2 \eta_2^k + 4c_{tc}^2 \eta_3^k \right) + \eta_4^k\\
&\le c_2 I_{4,h}^{k-1}  
+ \frac 1 {1-4c_{tc}^2} \left((2 {\tilde{\overline\theta}} + 4c_{tc}^2)r^k + 2 \eta_2^k + 4c_{tc}^2 \eta_3^k \right) + \eta_4^k\,.
\end{align*}
With the notation 
\be{ai}
  a^i\coloneqq   \frac 1 {1-4c_{tc}^2} \left((2 {\tilde{\overline\theta}} + 4c_{tc}^2)r^i + 2 \eta_2^i + 4c_{tc}^2 \eta_3^i \right) + \eta_4^i \quad\forall i\in \{1,2,\dots,k\}
\ee
there follows recursively 
\be{estI4_2}
  I_{4,h}^k \le c_2^{k}I_{4,h}^0+\sum_{j=0}^{k-1}c_2^j a^{k-j}\,.
\ee
Note that by the second part of \eqref{etacond2}, the second part of \eqref{etacond3} and the fact that $r^k\to 0$ as $k\to\infty$ (by definition 
of $r^k$) , we have $c_2^{k}I_{4,h}^0+\sum_{j=0}^{k-1}c_2^j a^{k-j} \to 0$ as $k\to\infty$. So, if the discrepancy principle never got active (i.e., $k_*=\infty$), the
sequence $(I_{4,h}^k)_{k\in\mathbb{N}}$ and therewith by assumption \eqref{etacond2} also
$(I_{3,h}^k)_{k\in\mathbb{N}}$ would be bounded by a sequence tending to zero as $k\to\infty$, which implies
that $I_{3,h}^k$ would fall below $\tau^2\delta^2$ for $k$ sufficiently large, thus yielding a contradiction. Hence
the stopping index $k_* < \infty$ is well-defined and finite.

\item[(iii):] With \eqref{noise}, \eqref{intcond}, \eqref{etacond2} and definition of $k_*$, we have 
\be{Fconv}
  \|F(\qoldkstern)-g\|_G \le \sqrt{I_3^{k_*}} + \delta  \le  \sqrt{I_{3,h}^{k_*} +\eta_3^{k_*}} + \delta
  \le \sqrt{(1+c_1)I_{3,h}^{k_*}} +\delta  \le (\sqrt{1+c_1}\tau +1) \delta \to 0
\ee
as $\delta \to 0$. Thus, due to (ii) \eqref{ineq0}  $\qoldkstern = \qhkminuseinsstern$ 
has a weakly convergent subsequence $(\qoldksterndeltal)_{l\in\mathbb{N}}$ and due to the weak sequential closedness of $F$ and \eqref{Fconv} the limit $q^*\in\mathcal{B}_{\rho}(q_0)$ of every weakly convergent subsequence is a solution to $F(q)=g$.

Strong convergence of $(\qoldksterndeltal)_{l\in\mathbb{N}}$ to $q^*$ follows from the standard argument
\[
\|\qoldksterndeltal-q^*\|_Q^2=\underbrace{\|\qoldksterndeltal -q_0\|_Q^2}_{\leq\|q^* -q_0\|_Q^2} + \|q^* -q_0\|_Q^2
-2\underbrace{\langle\qoldksterndeltal -q_0, q^* -q_0\rangle_Q}_{\to\|q^* -q_0\|_Q^2\ \mbox{\footnotesize by
weak convergence}}
\]
(with $\langle . , . \rangle_Q$ denoting the scalar product in $Q$), 
where we have used the fact that  in \eqref{ineq0} we can replace $q^\dag$ by $q^*$ since the latter
solves \eqref{OEq}.
\end{itemize}
\end{proof} 

\begin{rem}
Note that estimate \eqref{min} can alternatively be obtained by using stationarity instead of minimality of $\qk$ (which is equivalent by convexity): For all $dq\in Q$
\begin{eqnarray*}
0&=& \langle F'(\qoldk)(\qk-\qoldk) +F(\qoldk) -g^\delta, F'(\qoldk)(dq)\rangle_G
 + {\frac 1 {\beta_k}} \langle \qk -q_0,dq\rangle_Q
\end{eqnarray*}
(with $\langle . , . \rangle_G$ denoting the scalar product in $G$).  
With $dq=\qk-q^\dag$ this yields
\begin{eqnarray*}
0&=& \| F'(\qoldk)(\qk-\qoldk) +F(\qoldk) -g^\delta\|_G^2\\
&&-\langle F'(\qoldk)(\qk-\qoldk) +F(\qoldk) -g^\delta,
F'(\qoldk)(q^\dag-\qoldk) +F(\qoldk) -g^\delta\rangle_G\\
&& + {\frac 1 {\beta_k}} \|\qk -q_0\|_Q^2
- {\frac 1 {\beta_k}} \langle \qk -q_0,q^\dag-q_0\rangle_Q\,,
\end{eqnarray*}
hence by Cauchy-Schwarz and $ab\le \frac12 a^2+\frac12 b^2$ $\forall a,b\in \R$
\begin{eqnarray*}
\lefteqn{\| F'(\qoldk)(\qk-\qoldk) +F(\qoldk) -g^\delta\|_G^2
+ {\frac 1 {\beta_k}} \|\qk -q_0\|_Q^2}\\
&\leq& 
\|F'(\qoldk)(q^\dag-\qoldk) +F(\qoldk) -g^\delta\|_G^2+ {\frac 1 {\beta_k}} \|q^\dag-q_0\|_Q^2
\,.
\end{eqnarray*}
\end{rem}

\subsection{Convergence rates} \label{subsec_rates_redIRGN}

To prove convergence rates we will consider Hilbert space source conditions
\beq\label{source}
\exists s \in Q \mbox{ s.t. }q^\dag -q_0 = f(F'(q^\dag)^*F'(q^\dag)) s\,.
\eeq 
with 
$f\colon [0,\infty)\to[0,\infty)$ satisfying
\begin{eqnarray}
&&f(0)=0\,, \quad f^2 \mbox{ strictly monotonically increasing,}\nonumber\\ 
&&\phi \mbox{ convex, where }\phi\coloneqq (f^2)^{-1} \mbox{i.e., $\phi^{-1}(\lambda)=f^2(\lambda)$ and}
\label{fphipsi}\\ 
&&\Theta\colon \lambda\mapsto f(\lambda)\sqrt{\lambda}\mbox{ strictly monotonically increasing.}\nonumber
\end{eqnarray}
Examples of functions $f$ satisfying \eqref{fphipsi} are H\"older type $f(\lambda)=\lambda^\nu$, $\nu>0$
or logarithmic type $f(\lambda)=\ln(\frac{1}{\lambda})^{-p}$, $\lambda\in(0,1/e]$, $p>0$ functions.

If \eqref{source} holds, then we have by Jensen's inequality
\begin{eqnarray}\label{source1}
|\langle q^\dag - q_0, q -q^\dag \rangle_Q| &=& |\langle s,
f(F'(q^\dag)^*F'(q^\dag))(q-q^\dag)\rangle_Q| \nonumber\\
&\leq& \|s\|_Q \, \|q-q^\dag\|_Q f\left( \frac{\|F'(q^\dag)(q-q^\dag)\|_G^{2}}{\|q-q^\dag\|_Q^2}\right)\,.
\end{eqnarray}

Using estimate \eqref{ineq0} from the proof of Theorem \ref{conv}, as well as the definition of the
stopping index according to the discrepancy principle, we can therefore make use of Theorem 1 in \cite{KKV10}
to obtain

\begin{thm} \label{rates}
Let the conditions of Theorem \ref{conv} and additionally
the source condition \eqref{source} for some function $f$ with \eqref{fphipsi} be fullfiled.

Then there exists a $\bar{\delta}>0$ and a constant $\bar{C}>0$ independent of $\delta$ such that
for all $\delta\in (0,\bar{\delta}]$
\begin{equation}\label{rates0}
\|\qhkstern -q^\dag\|_Q^2
\leq \frac{\bar{C}^2\delta^2}{\Theta^{-1} \left(\tfrac{\bar{C}}{2\|s\|_Q}\delta\right)}
= 4\|s\|^2 f^2(\Theta^{-1}(\tfrac{\bar{C}}{2\|s\|_Q}\delta))
\end{equation}
where $\Theta(\lambda)\coloneqq f(\lambda)\sqrt{\lambda}$. 
\end{thm}
\begin{proof}
The assertion follows from Theorem 1 in \cite{KKV10} using the estimate \eqref{ineq0} and 
\begin{eqnarray*}
\|F(\qhkstern)-g^\delta\|_G&\leq&
\sqrt{I_{3,h}^{k_*}+\eta_3^{k_*}}\\
&\leq& \sqrt{1+c_1}\sqrt{I_{3,h}^{k_*}}
\leq \sqrt{1+c_1}\tau\delta\,.
\end{eqnarray*}
\end{proof}
\begin{rem}
If $f$ satisfies the condition
\beq\label{nule12}
t \;\mapsto \; \frac{f(t)}{\sqrt{t}} \;\;\mbox{ monotonically decreasing},
\eeq
then  for all $C>0$ the inequality
\beq\label{quof}
f(\Theta^{-1}(Ct)) \le \max\{\sqrt{C},1\}\,f(\Theta^{-1}(t))\qquad(t \ge 0)
\eeq
holds, which implies that we can conclude from \eqref{rates0} the optimal rates
\begin{equation}\label{rates1}
\|\qhkstern -q^\dag\|_Q\leq C\left(f^2(\Theta^{-1}(\delta))\right) 
= C\left(\frac{\delta^2}{\Theta^{-1}(\delta)}\right).
\end{equation}
The restriction \eqref{nule12} corresponds to the typical saturation phenomenon of Tikhonov regularization in
combination with the discrepancy principle {, see e.g., \cite{EHNBuch}.}
 \end{rem}

\subsection{Computation of the error estimators}\label{subsec_estimatorsIRGNM}

The computation of the error estimators $\eta_1^k$, $\eta_2^k$, $\eta_3^k$ and $\eta_4^k$ is done
similarly to \cite{GKV}. The only difference lies in the fact that in $I_1^k$ we have three
variables subject to discretization, namely $q$, $\uold$ and $w$ instead of only two ($q$ and $u$)
as usual, which leads to the following error estimators. In this section we omit the iteration
index $k$ for simplicity.

\subsubsection{Error estimator for $I_1$}\label{subsec_I1estimatorIRGNM}

Since the dependence on $\beta$ is not important for error estimation, we neglect $\beta$ as argument and
consider
\[
   I_1(q,\uold,w)=\|C'(\uold)(w)+C(\uold)-g^\delta\|_G^2 + {\frac 1 {\beta}} \|q-q_0\|_Q^2
\]
and define the Lagrange functional
\begin{align*}
   L(q,\uold,w,v,\vold)&\coloneqq  I_1(q,\uold,w)\\
   &\quad +A_u'(\qold,\uold)(w)(v)+A_q'(\qold,\uold)(q-\qold)(v)\\
   &\quad +A(\qold,\uold)(\vold)-f(\vold)\,.
\end{align*}

\begin{prop}\label{prop_I1estimator}
Let $X = Q \times V \times V \times W \times W$ and $X_h = Q_h\times V_h \times V_h \times W_h
\times W_h$. Let $x=(q,\uold,w,v,\vold)\in X$ be a stationary point of $L$, i.e.
\[
  \xh \in X_h\colon \qquad L'(x)(dx) = 0 \qquad \forall dx\in X
\]
and let $\xh=(\qh,\uoldh,\wh,\vh,\voldh)\in X_h$ be a discrete stationary point of $L$, i.e.
\beq \label{eq_statpointdiscrete}
  L'(\xh)(dx)=0 \qquad \forall dx \in X_h\,.
\eeq
Then there holds
\[
  I_1(q,\uold,w)-I_1(\qh,\uoldh,,\wh)=\frac 12 L'(\xh)(x-\tilde{x}_h)+R\,,
\]
for an arbitrary $\tilde{x}_h\in X_h$ and
\[
  R=\frac 12 \int_0^1 L'''(x_se_x)(e_x,e_x,e_x)s(s-1)\ ds
\]
with $e_x\coloneqq x-\xh$.
\end{prop}

\begin{proof}
cf. \cite{GKV} and \cite{BeckerRannacher}.
\end{proof}

Explicitly such stationary points can be computed by solving the equations
\begin{align}
\uold\in V\colon \qquad & A(\qold,\uold)(d\vold)=f(d\vold)\quad \forall d\vold\in W\label{eq:I1equation1}\\
w\in V: \qquad & A_u'(\qold,\uold)(w)(dv) = -A_q'(\qold,\uold)(q-\qold)(dv)\quad \forall dv\in W\\
v\in W:\qquad & A_u'(\qold,\uold)(dw)(v) = 
-I_{1,w}'(q,\uold,w)(dw) 
\quad \forall dw\in V\\
\vold\in W:\qquad & A_u'(\qold,\uold)(du)(\vold) =
   -I_{1,\uold}'(q,\uold,w)(du)\\
   &\hspace*{39mm} - A_{uu}''(\qold,\uold)(w,du)(v)\\
   & \hspace*{39mm}-A_{qu}''(\qold,\uold)(q-\qold,du)(v)\quad \forall du\in V\\
q\in Q:\qquad & 
I_{1,q}'(q,\uold,w)(dq) 
= - A_q'(\qold,\uold)(dq)(v)\quad \forall dq\in Q\,.\label{eq:I1equation5}
\end{align}

%
and their discrete counterparts.

Obviously, we do not actually compute continuous stationary points, but (as in \cite{GKV}) we
choose $\tilde x_h=i_h x$ with a suitable interpolation operator
$i_h\colon  X \to X_h$ and approximate the interpolation error using an operator
$\pi_h\colon  X_h\to\tilde X_h$ with $\tilde X_h\neq X_h$, such that $x-\pi_h x_h$ has a better local
asymptotical behavior than $x-i_h x$. Then the error estimator $\eta_1$ for $I_1$ can be computed as
\[
 I_1 - I_{1,h} =   I_1(q,\uold,w)-I_1(\qh,\uoldh,\wh)  \approx  \frac 12 L'(\xh)(\pi_h\xh-\xh) = \eta_1
\]
(cf. \cite{BeckerRannacher}).

{
\begin{rem}
 Please note that the equations \eqref{eq_statpointdiscrete} /\eqref{eq:I1equation1}-\eqref{eq:I1equation5} are solved anyway in the process of solving the optimization problem 
 \eqref{eq_IRGNM_optprob}-\eqref{PDElin}.
\end{rem}
}
\subsubsection{Error estimator for $I_2$}\label{subsec_I2estimatorIRGNM}

We consider
\[
   I_2(\uold,w)= \|C'(\uold)(w)+C(\uold)-g^\delta\|_G^2
\]
and for $x_1\coloneqq (q_1,\uoldeins,w_1,v_1,\voldeins) \in X$ we define the Lagrange functional
\[
   M(x,x_1)\coloneqq  I_2(\uold,w)+L'(x)(x_1)\,.
\]
Then there holds a similar result to Proposition \ref{prop_I1estimator} for the difference
$I(\uold,w)-I(\uoldh,\wh)$ for stationary points $y=(x,x_1)\in X\times X$ and
$\yh=(\xh,\xeinsh)\in X_h\times X_h$ of $M$ (cf.~\cite{BeckerVexler}). Such a discrete stationary point $y_h$ can
be computed by solving the equations \eqref{eq_statpointdiscrete}/\eqref{eq:I1equation1}-\eqref{eq:I1equation5}  and
\beq\label{eq_statpointM}
  \xeinsh\in X_h\colon \qquad L''(\xh)(\xeinsh,dx) = -I_{2,w}'(\uoldh,\wh)(dw)-I_{2,\uold}'(\uoldh,\wh)(d\uold) \qquad
\forall   dx\in X_h\,,
\eeq
(where $dx=(dq,d\uold,dw,dv,d\vold)$). The error estimator $\eta_2$ for $I_2$ can then be computed by
\[
  I_2-I_{2,h} = I_2(\uold,w)-I_2(\uoldh,\wh)\approx \frac 12 M'(\yh)(\pi_h\yh-\yh) = \eta_2\,.
\]
{
\begin{rem}
 To avoid the computation of second order information in \eqref{eq_statpointM} we would like to refer to \cite{BeckerVexler}, where 
 \eqref{eq_statpointM} is replaced by an approximate equation of first order. 
\end{rem}
}
\subsubsection{Error estimator for $I_3$}\label{subsec_I3estimatorIRGNM}

For $I_3$ we again proceed similarly to the sections \ref{subsec_I1estimatorIRGNM} and
\ref{subsec_I2estimatorIRGNM}, i.e. we consider
\[
   I_3(\uold)= \|C(\uold)-g^\delta\|_G^2
\]
and define the Lagrangian
\[
   N(x,x_2)\coloneqq I_3(\uold)+L'(x)(x_2)\,.
\]
As there holds again a similar results to Proposition \ref{prop_I1estimator}, we compute
a discrete stationary point $\chi_h=(\xh,\xzweih)\in X_h\times X_h$ of $N$
by solving the equations \eqref{eq_statpointdiscrete}/\eqref{eq:I1equation1}-\eqref{eq:I1equation5}  and
\beq\label{eq_statpointN}
\xzweih\in X_h\colon \qquad L''(\xh)(\xzweih,dx) = - I_3'(\uoldh)(d\uold) \qquad \forall dx\in X_h\,,
\eeq
and compute the error estimator for $I_3$ as
\[
  I_3-I_{3,h}= I_3(\uold)-I_3(\uoldh) \approx  \frac 12 N'(\chi_h)(\pi_h \chi_h-\chi_h) = \eta_3\,.
\]


\subsubsection{Error estimator for $I_4$}\label{subsec_I4estimatorIRGNM}

{
Different to the other error estimates, the bound on the error in $I_4$ only appears in connection with the very weak assumption
$\eta_4^{k}\to 0$ as $k\to\infty$, which may be satisfied in practice without refining explicitly with respect to $\eta_4$, but simply, 
by refining with respect to the other error estimators $\eta_1,\eta_2$, and especially $\eta_3$. Another way to make sure that $\eta_4^{k}\to 0$ as $k\to\infty$, 
is, of course, to refine globally every now and then, although this is admittedly, not a very efficient solution.

If one doesn't want to rely on such practically motivated speculations and actually wants to compute an error estimator for $I_4$, one has to   
include the decoupled constraint
\[
  A(q,u)(v)=f(v)\qquad \forall v\in W
\]
in the definition of the Lagrangian $L$ in subsection \ref{subsec_I1estimatorIRGNM}. In that case we 
redefine the Lagrange functional $L$ in subsection \ref{subsec_I1estimatorIRGNM} as
\begin{align*}
   L(q,\uold,w,v,\vold,u,z)&\coloneqq  I_1(q,\uold,w)\\
   &\quad +A_u'(\qold,\uold)(w)(v)+A_q'(\qold,\uold)(q-\qold)(v)\\
   &\quad +A(\qold,\uold)(\vold)-f(\vold)\\
   &\quad +A(q,u)(z)-f(z)\,.
\end{align*}
and the spaces $X\coloneqq Q\times V\times V\times W\times W\times V \times W$ and $X_h\coloneqq Q_h\times
V_h\times V_h\times W_h\times W_h\times V_h \times W_h$. 
Then we consider
\[
   I_4(u)\coloneqq \|C(u)-g^\delta\|_G^2
\]
and define the auxiliary Lagrange functional
\[
   K(x,x_3)\coloneqq I_4(u)+L'(x)(x_3)
\]
for $x,x_3\in X$. Then again (as in the subsections \ref{subsec_I1estimatorIRGNM},
\ref{subsec_I2estimatorIRGNM} and \ref{subsec_I3estimatorIRGNM}) we could estimate the difference $I_4(u) -
I_4(\uh)$ by computing a discrete stationary point $\xi_h = (x_h,\xdreih)$ of $K$, that means we would solve 
he equations
\eqref{eq_statpointdiscrete}/\eqref{eq:I1equation1}-\eqref{eq:I1equation5}  and
\[
  \xdreih \in X_h\colon \qquad L''(\xh)(\xdreih,dx) = -I'_4(\uh)(du) \forall dx=(dq,du,dz)\in X_h\,,
\]
with $X_h\coloneqq Q_h\times V_h\times V_h\times W_h\times W_h\times V_h \times W_h$ and compute the error estimator
 $\eta_4$ for $I_4$ by
\[
  I_4-I_{4,h} = I_4(u)-I_4(\uh) \approx \frac 12 K'(\xi_h)(\pi_h \xi_h-\xi_h) = \eta_4\,.
\]
}

\subsection{Algorithm}\label{subsec_algorithm}

As mentioned and justified in Subsection \ref{subsec_I4estimatorIRGNM}, we 
neglect $\eta_4^k$ and the condition $\eta_2^k\to 0$ and $\eta_3^k\to 0$ as $k\to\infty$ from
\eqref{etacond2} in the following algorithm. 

In order to verify the condition \eqref{etacond1} more easily, we split \eqref{etacond1} into 
\be{(A)}
  \eta_1^k \le c_4 I_{3,h}^k \quad \text{with} 
  \quad c_4 = \frac 12 \left({\tilde{\underline\theta}} - 
    2 \left(c_{tc}^2 + \frac{(1+c_{tc})^2}{\tau^2}\right)\right)
\ee
and 
\be{(C1)}
  \eta_3^k \le c_5 I_{3,h}^k\quad \text{with} 
  \quad c_5 = \frac 1 {4c_{tc}^2}  \left({\tilde{\underline\theta}} - 
    2 \left(c_{tc}^2 + \frac{(1+ c_{tc})^2}{\tau^2}\right)\right)\,.
\ee

Additionally we combine the inequality in \eqref{etacond2}, 
the first inequality in \eqref{etacond3} and \eqref{(C1)}, since there holds 
\[
  I_{3,h}^k \le I_3^k+\eta_3^k \quad \text{ and } \quad I_{4,h}^{k-1} \ge I_4^{k-1} - \eta_4^{k-1}\,,  
\]
such that the condition
\be{eta3eta4condsimple}
  \eta_3^k + (1+c_3)\eta_4^{k-1} \le (1+c_3) I_4^{k-1} - I_3^k + r^k
\ee
implies the first inequality in \eqref{etacond3}.
As mentioned in Remark \ref{remI3I4}, $I_3^k$ and $I_4^{k-1}$ and  $I_{3,h}^k$ and $I_{4,h}^{k-1}$ only differ in the discretization level, 
which motivates the assumption that for small $h$, we have $I_3^k \approx I_4^{k-1}$ and $\eta_4^{k-1} \approx \eta_3^k$, such that 
instead of \eqref{eta3eta4condsimple} we check whether
\be{eta3condsimple}
  \eta_3^k \le \frac {c_3} {2(1+c_3)} I_{3,h}^k + \frac {r^k} {2(1+c_3)} \,.  
\ee
Thus, as a combination of the inequality in \eqref{etacond2}, \eqref{eta3condsimple}
and \eqref{(C1)}, we formulate
\be{(C)}
  \eta_3^k \le \min\left\{c_1,c_5,\frac {c_3}{2(1+c_3)}\right\} I_{3,h}^k\,.
\ee

\begin{algorithm}{Reduced form of discretized IRGNM}\label{alg_IRGNM}
  \begin{algorithmic}[1]
    \STATE Choose $\tau$, $\tau_\beta$, $\tilde\tau_\beta$, $\uttheta$, $\ottheta$ such that $ 0 < \uttheta \le \ottheta <1$
     and \eqref{taucond} holds , $\tilde\theta = (\uttheta + \ottheta) / 2$ 
         and $\max\{1\,,{\tilde\tau}_\beta\} <\tau_\beta\le \tau$ 
and choose the constants $c_1$, $c_2$ and $c_3$, such that the second part of
	  \eqref{etacond3} is fulfilled. 
    \STATE Choose a discretization $h=h_0$ and starting value $\qhnull = \qhnullnull$ 
          (not necessarily coinciding with $q_0$ in the regularization term)
	  and set $\qoldnull= \qhnullnull$.
    \STATE Determine $\uoldnull = \uoldhnullnull$, $I_{3,h}^0 = I_{3,{h_0}}^0$ 
	and $\eta_{3}^0 = \eta_{3,h_0}^0$
	   by applying Algorithm \ref{alg_I3Evaluation} with $m=0$ (and $h=h_0$).
    \STATE  Set $h^1_0 = h_0$.
    \WHILE {\eqref{(C)} is violated} 
      \STATE Refine grids according to the error estimator $\eta_3^0$, such that we obtain a finer discretization
      $h^1_0$.
      \STATE Determine $\uoldnull = \uoldheinsnullnull$, $I_{3,h}^0 = I_{3,{h^1_0}}^0$ and 
      $\eta_{3}^0 = \eta_{3,h^1_0}^0$
	   by applying Algorithm \ref{alg_I3Evaluation} with $h=h^1_0$ and $m=0$.
    \ENDWHILE
     \STATE Set $k=0$ and $h=h^1_0$ (possibly different from $h_0$).
     
    \WHILE {$I_{3,h}^k\ge \tau^2\delta^2$}
	   \STATE Set $h = h^1_k$. 
           \STATE 
		With $\qoldk$, $\uoldk$ fixed, apply Algorithm \ref{alg:GKVAlgorithm} 
		starting with the current mesh $h (= h^1_k)$		
		to obtain a regularization parameter $\beta_k$ and a possibly different discretization $h^2_k$
 		such that \eqref{inexNewton} holds and the corresponding $\whk = \whzweik$, $\qhk = \qhzweik$.
 		\STATE Set $h=h^2_k$.
		\STATE Evaluate error estimator $\eta_1^k = \eta_1^k(h^2_k)$.
		\STATE Set $h^3_k = h^2_k$.
		\WHILE {\eqref{(A)} is violated} 
		  \STATE Refine grids according to the error estimator $\eta_1^k$, such that we obtain 
		  a finer discretization $h^3_k$.
		  \STATE Set $h=h^3_k$.
		  \STATE With $\qoldk$ and $\uoldk$ fixed, determine $\qhk = \qhdreik$ and $\whk = \whdreik$
		  by solving \eqref{optprobinalg} 
		\ENDWHILE
		\STATE Determine $\uhk = \uhdreik \in V_h=V_{h^3_k}$ by solving
		    \[
			A(q^k_{h},u^k_{h})(v)=f(v) \qquad \forall v\in W_{h}\,.
		    \]
          	  \STATE Set $\qoldkplus=\qhk(=\qhdreik)$ and 
          	  $\uoldkplus (= \uoldhdreikplus)=\uhk(=\uhdreik)$.
          	  \STATE Evaluate $I_{3,h}^{k+1} = I_{3,h^3_k}^{k+1}$ according to \eqref{eq_IRGNM_I1234kdiscrete} 
		  and the error estimator $\eta_{3}^{k+1} = \eta_{3}^{k+1}({h^3_k})$.
		  Set $h^4_k=h^3_k$.
		  \WHILE {\eqref{(C)} is violated}
		    \STATE Refine grid according to the error estimator $\eta_3^{k+1}$, such that we obtain a finer 
		    discretization $h^4_k$.
		    \STATE Determine $\uoldkplus = \uoldhvierkplus$, 
		    $I_{3,h}^{k+1} = I_{3,{h^4_k}}^{k+1}$ and $\eta_{3}^{k+1} = \eta_{3,h^4_k}^{k+1}$
		    by applying Algorithm \ref{alg_I3Evaluation} with $m=k+1$ and $h=h^4_k$.
		  \ENDWHILE
		  \STATE Set $h^1_{k+1} = h^4_k$  (i.e. use the current mesh as a starting mesh for the next iteration)
		  \STATE Set $k=k+1$
	\ENDWHILE
\end{algorithmic}
\end{algorithm}

\begin{algorithm}{Evaluation of $I_3^m$}\label{alg_I3Evaluation}
  \begin{algorithmic}[1]
    \STATE Determine
	 \[
	    \uold^m \in V_h: 
	    \qquad A(\qold^m,\uold^m)(v) = f(v) \qquad \forall v\in W_h\,.
	 \]
    \STATE Evaluate $I_{3,h}^m$ according to \eqref{eq_IRGNM_I1234kdiscrete}. 
    \STATE Evaluate error estimator $\eta_3^m$.
\end{algorithmic}
\end{algorithm}

\begin{algorithm}{Inexact Newton method for the determination of a regularization parameter for the IRGNM subproblem from \cite{GKV}}\label{alg:GKVAlgorithm}
  \begin{algorithmic}[1]
    \STATE Set $\delta_\beta = \sqrt{\tilde\theta I_{3,h}^{k-1}} / \tau_\beta$.    
    \STATE Compute a Lagrange triple $\xh=(\qh,\wh,\zh)$ to 
		    \be{optprobinalg}
			\min_{(q,w)\in Q_{h} \times V_{h}}
			\|C'(\uoldk)(w)+C(\uoldk)-g^\delta\|_G^2+{\frac 1 {\beta_k}}\|q-q_0\|_Q^2
		    \ee
		    \[
		      \text{s.t.}\qquad A'_u(\qoldk,\uoldk)(w,v)+A'_q(\qoldk,\uoldk)(q-\qoldk)(v)
		      + A(\qoldk,\uoldk)(v) - f(v)=0
		      \qquad \forall v\in W_{h}\,.
		    \]
    \STATE Evaluate $i_{h} = I_{h,2}^k = \|C'(\uoldk)(w_h)+C(\uoldk)-g^\delta\|_G^2$.
    \WHILE{$i_{h} > (\tau_\beta^2+\frac {\tilde \tau_\beta^2} 2)  \delta_\beta^2$}
    \STATE Evaluate $i'_{h}$ (cf. \cite{GKV}).
    \STATE Evaluate error estimator for $i(\beta) = I(w(\beta))$ with 
	  $I\colon  w\mapsto I_2(\uoldk,w)$ (cf. \cite{GKV}).
    \STATE Evaluate error estimator for $i'(\beta) = \frac {d}{d\beta} I(w(\beta))$ (cf. \cite{GKV}).
    \WHILE {accuracy requirements (cf. \cite{GKV}) are violated}
    	\STATE  Refine with respect to the corresponding error estimator. \label{step_GKV_refine1}
        \STATE Compute a Lagrange triple $\xh=(\qh,\wh,\zh)$ to \eqref{optprobinalg}.
        \STATE Evaluate $i_{h} = I_{h,2}^k = \|C'(\uoldk)(w_h)+C(\uoldk)-g^\delta\|_G^2$.
        \STATE Evaluate $i'_{h}$ (cf. \cite{GKV}).
        \STATE Evaluate error estimator for $i(\beta) = I(w(\beta))$ with 
	  $I\colon  w\mapsto I_2(\uoldk,w)$ (cf. \cite{GKV}).
        \STATE Evaluate error estimator for $i'(\beta) = \frac {d}{d\beta} I(w(\beta))$ (cf. \cite{GKV}).
    \ENDWHILE
    \STATE Update $\beta$ according to an inexact Newton method (cf. \cite{GKV}) $\beta\leftarrow\beta-\frac{i_h}{i'_h}$.
    \STATE Compute a Lagrange triple $\xh=(\qh,\wh,\zh)$ to \eqref{optprobinalg}.
    \STATE Evaluate $i_{h} = I_{h,2}^k = \|C'(\uoldk)(w_h)+C(\uoldk)-g^\delta\|_G^2$.
    \ENDWHILE
\end{algorithmic}
\end{algorithm}

\begin{rem}\label{rem:algoGKV}
Algorithm \ref{alg:GKVAlgorithm} corresponds to the algorithm from \cite{GKV}
with the following replacements:\\
\begin{tabular}{l|l}
	in \cite{GKV} & here\\
	\hline
	$q$ & $q-q_0$\\
	$T$ & $F'(\qoldk)$\\
	$g^\delta$ & $g^\delta-F(\qoldk)+F'(\qoldk)(\qoldk)$\\
	$\tau^2\delta^2$ & $\tilde\theta I_{3,h}^k$\\
	$(\tau-\tilde{\tau})^2\delta^2$ & ${\tilde{\underline\theta}} I_{3,h}^k$\\
	$(\tau+\tilde{\tau})^2\delta^2$ & ${\tilde{\overline\theta}} I_{3,h}^k$
\end{tabular}
\end{rem}


With respect to loops and the solution of PDEs and optimization problems, the algorithm has the following form. 
(We do not display the refinement loops on lines 5, 15, 22 of Algorithm \ref{alg_IRGNM} and on line 8 of Algorithm \ref{alg:GKVAlgorithm} but only the iteration loops.)
\begin{algorithm}{Loops in reduced form of discretized IRGNM}\label{alg_loopsIRGNM}
  \begin{algorithmic}[1]
    \WHILE{$\cdots$ (Newton iteration)}
      \STATE Apply algorithm from \cite{GKV}, i.e.
      \WHILE{$\cdots$ (Iteration for $\beta_k$)}
	  \STATE Solve linear-quadratic optimization problem (i.e. solve linear PDE).
	\STATE Update $\beta$ and refine eventually.
      \ENDWHILE
      \STATE Solve nonlinear PDE.
    \ENDWHILE
 \end{algorithmic}
\end{algorithm}

In contrast with the nonlinear Tikhonov method 
\[
\min_{q\in Q}\|F(q)-\gdel\|_G^2+\frac{1}{\beta} \|q-q_0\|_Q^2\,.
\]
investigated in \cite{KKV10} (cf. algorithm \ref{alg_loopsINMnonlinear} below), we have one additional loop, but we
only have to solve a linear-quadratic optimization problem instead of a nonlinear problem. On the other hand,
we still have to solve 
(at least) one nonlinear PDE 
in each outer loop. For this reason we doubt whether algorithm \ref{alg_IRGNM} pays off
with respect to computation time as compared to the method in \cite{KKV10}. Therefore we do  not implement this algorithm, but consider more efficient modifications in \cite{KKV13} {(part II of this paper)}.

\begin{algorithm}{Loops in Inexact Newton Method (for nonlinear problems)}\label{alg_loopsINMnonlinear}
  \begin{algorithmic}[1]
    \WHILE{$\cdots$ (Iteration for $\beta$)}
	\STATE Solve nonlinear optimization problem (i.e. solve nonlinear PDE).
      \STATE Update $\beta$ and refine eventually.
    \ENDWHILE    
 \end{algorithmic}
\end{algorithm}

\subsection{Extension to more general data misfit and regularization terms}\label{subsec_generalize}

Motivated by the increasing use of nonquadratic, non-Hilbert space misfit and regularization terms for modelling, e.g., sparsity of the solution, or non-Gaussian data noise (cf., e.g., \cite{PoeschlDiss,Flemming2010} for Tikhonov regularization, and \cite{HohageWerner} for the IRGNM), we now extend our results to a more general setting. 
To this purpose we consider a more general version of \eqref{eq_IRGNM_minJ1Discrete}:
\be{eq_IRGNM_minJ1DiscreteGeneral}
\min_{q \in Q} 
\mathcal{T}_\beta (q)\coloneqq \cS(F'(\qoldk)(q- \qoldk)+F(\qoldk),g^\delta)+{\frac  1 \beta}\cR(q)
\eeq 
with quantities of interest (cf. \eqref{eq_IRGNM_I1234})
\beq \label{eq_IRGNM_I1234kGeneral}
\begin{aligned}
\Ieinsk & \coloneqq  \cS(F'(\qoldk)(q^k- \qoldk)+F(\qoldk),g^\delta)
			+{\frac  1 \beta}\cR(q^k)\\ 
\Izweik & \coloneqq  \cS(F'(\qoldk)(q^k- \qoldk)+F(\qoldk),g^\delta)\\
\Idreik &\coloneqq  \cS(F(\qoldk),g^\delta)\\
\Ivierk & \coloneqq  \cS(F(q^k),g^\delta)
\end{aligned}
\eeq
and its discrete counterparts (cf. \eqref{eq_IRGNM_minJ1Discretereduced})
\be{eq_IRGNM_minJ1DiscretereducedGeneral}
\min_{q \in Q_h} 
\mathcal{T}_{\beta,h}(q) \coloneqq  \cS(F_h'(\qoldk)(q- \qoldk)+F_h(\qoldk),g^\delta)+{\frac  1 \beta}\cR(q)
\eeq
with
\beq \label{eq_IRGNM_I1234kdiscreteGeneral}
\begin{aligned}
\Ieinshk & \coloneqq  \cS(F_{h_k}'(\qoldk)(q^k_{h_k}- \qoldk)+F_{h_k}(\qoldk),g^\delta)
			+{\frac  1 \beta}\cR(q^k_{h_k})\\ 
\Izweihk & \coloneqq  \cS(F_{h_k}'(\qoldk)(q^k_{h_k}- \qoldk)+F_{h_k}(\qoldk),g^\delta)\\
\Idreihk &\coloneqq  \cS(F_{h_k}(\qoldk),g^\delta)\\
\Ivierhk & \coloneqq  \cS(F_{h_k}(q^k_{h_k}),g^\delta)
\end{aligned}
\eeq
(cf. \eqref{eq_IRGNM_I1234kdiscrete}).

The data misfit and regularization functionals $\cS$ and $\cR$ should satisfy 
\begin{ass}\label{ass_IRGNM_SRass}
Let $\cS\colon G\times G \to \R $ and $\cR\colon Q\to\R$ have the following properties: 
\begin{enumerate}
\item The mapping $y\mapsto \cS(y ,g^\delta)$ is convex. \label{item_IRGNM_Sconvex}
\item $\cS$ is symmetric, i.e. $\cS(y,\tilde{y})=\cS(\tilde{y},y)$ for all $y,\tilde y \in G$. \label{item_IRGNM_Ssymm}
\item $\cS$ is positive definite, i.e. $\cS(y,\tilde{y})\geq0$ and $\cS(y,y) = 0$ for all $y,\tilde y\in G$. \label{item_IRGNM_Sdef}
\item For all $y,\tilde y, \widehat y \in G$ there exists a constant $c_{\cS}$ such that
$\cS(y,\tilde{y})\leq c_{\cS}(\cS(y,\hat{y})+\cS(\hat{y},\tilde{y}))$. \label{eq_IRGNM_Striangleineq}
\item The regularization operator $\cR$ is proper (i.e. the domain of $\cR$ is non-empty)
 and convex. \label{eq_IRGNM_Pproperconvex}
\end{enumerate}
\end{ass}
where the domain of an operator ${\cal R}\colon  M \to \R $ should be understood as
\[
	\mathcal{D}({\mathcal R})\coloneqq \{m\in M|\ {\cal {R}}(m) \neq\infty\}\,.
\]

\begin{rem}
In fact, it suffices to require $\cS(y,y) = 0$ only for $y=g$, i.e. for the exact data in Item \ref{item_IRGNM_Sdef} in Assumption \ref{ass_IRGNM_SRass}, but 
since Item \ref{item_IRGNM_Sdef} is a more ''natural`` assumption in terms of general operator properties, we stick with the stronger assumption 
 Item \ref{item_IRGNM_Sdef}.
\end{rem}

We refer once more to \cite{HohageWerner} where convergence and convergence rates for the IRGNM have already been established in an even more general (continuous) setting and mention that we here consider a somewhat simpler situation with stronger assumptions
on $\cS$, $\cR$, since our main intention is to demonstrate transferrability of the adaptive discretization concept. Moreover note, that we rely on a different choice of the regularization parameter here.
The results obtained here will allow us to easily establish convergence rates results for an exact penalty formulation of an all-at-once  formulation of the IRGNM in \cite{KKV13} {(part II of this paper)}.

Although we will, again, restrict ourselves to Hilbert spaces in the next sections, at this point
we discuss convergence in a Banach space setting to emphasize the generality of the subsequent results. To this purpose we
introduce the \emph{Bregman distance} 
\be{eq_IRGNM_bregmandistance} 
D_\cR^\xi(q,\ol q)
\coloneqq \cR(q)-\cR(\ol q)-\langle \xi,q-\ol q\rangle_{Q^*,Q}
\ee
with some $\xi \in \partial\cR(q)\subset Q^*$, which coincides with 
$\frac12 \|q-q^\dagger\|_Q^2$ for $\cR(q)=\frac12 \|q-q_0\|_Q^2$ and $\xi = q^\dagger-q_0$ in a Hilbert space $Q$.

Well-definedness (i.e. for every $\gdel \in G$ and $\beta_k>0$ there exists a solution $q_{h_k}^k$ to \eqref{eq_IRGNM_minJ1DiscretereducedGeneral})
and stable dependence on the data (i.e. for every fixed $\beta_k>0$ the solution $q_{k}^{h_k}$ depends continuously on $\gdel$) can be shown under
the following assumptions (cf., e.g., Assumption 1.32 in \cite{PoeschlDiss} or Remark 2.1 in \cite{HohageWerner})
\begin{ass}
\label{as:ip:bregman}
\begin{enumerate}
\item $Q$ and $G$ are Banach spaces, with which there are associated topologies $\tau_Q$ and
      $\tau_G$, which are weaker than the norm topologies.
\item The mapping $y\mapsto \cS(y ,g^\delta)$ is sequentially lower semi-continuous with respect to $\tau_G$.
\item $F'(\qoldk)\colon  Q \rightarrow G$ is continuous with respect to the topologies
      $\tau_Q$ and $\tau_G$.
\item ${\cal R} \colon  Q \to [0,+\infty]$ is proper, convex and $\tau_Q$-lower semi continuous.
\item ${\cal D}\coloneqq \mathcal{D}(F) \cap \mathcal{D}(\cR) \neq \emptyset$ is closed with respect to $\tau_Q$.
\item For every $C > 0$ the set
      \begin{equation}
      \label{ip:m}
      {\cal C}(C)\coloneqq \{ q \in {\cal D} \colon  \cR(q) \leq C\}\,,
      \end{equation}
      is $\tau_Q$-sequentially compact in the following sense:
      every sequence $(q_n)_{n\in\N}$ in ${\cal C}(C)$
      has a subsequence, which is convergent in $Q$ with respect to the $\tau_Q$-topology.
\end{enumerate}
\end{ass}

For well-definedness of the a posteriori chosen regularization parameter $\beta_k$ we refer to Lemma 1 and
Theorem 3 in \cite{KSS09}.

\begin{rem}\label{remHilbertspacecase}
For Hilbert spaces $Q$ and $G$ and the choice 
$S(y,\tilde y)\coloneqq \frac12\normGklein{y-\tilde y}^2$ and $\cR(q)\coloneqq \frac12\|q-q_0\|_Q^2$ 
Assumption \ref{ass_IRGNM_SRass} and Assumption \ref{as:ip:bregman}
are obviously fulfilled. {As for examples in a real Banach spaces setting, we refer to \cite{KSS09,KH10,PoeschlDiss}.} 
\end{rem}

Consistently the conditions \eqref{Fseqclosed} and \eqref{tangcone} on $F$ are generalized to
the following two assumptions. 

\begin{ass}\label{FseqclosedS}
Let the reduced forward operator $F$ be continuous with respect to $\tau_Q$,$\tau_G$ and satisfy
\[
    (q_n\stackrel{\tau_Q}{\to} q \ \wedge \ \cS(F(q_n),g)\to 0\ \Rightarrow \
    (q\in \mathcal{D}(F) \ \wedge \ F(q)=g)\\
\]%
for all $(q_n)_{n\in\N}\subseteq Q$\,.
\end{ass}

\begin{ass}\label{tangconeS}
Let the generalized tangential cone condition
\[
    \cS(F(q),F(\bar{q})+F'(q)(q-\bar{q}))\leq c_{tc}^2\cS(F(q),F(\bar{q}))
\]
hold for all $q,\bar{q}\in Q$ in a neighborhood 
of $q_0$ 
for some $0<c_{tc}<1$.
\end{ass}

Moreover, the source condition \eqref{source} is replaced by Assumption \ref{source2}.

\begin{ass}\label{source2}
Let the multiplicative variational inequality 
\[
  |\langle \xi, q -q^\dag\rangle_{Q^*,Q}|
  \le c D_\cR^\xi(q,q^\dag)^{1/2} 
  f\left( \frac{\cS(F(q),F(q^\dag))}{D_\cR^\xi(q,q^\dag)}\right)
\]
for all $q\in \mathcal{D}(F)$ hold \,.
  \end{ass}

Based on this groundwork, we can now formulate a convergence theorem similar to 
Theorem \ref{conv}:

 \begin{thm}\label{convS}
 Let Assumption \ref{as:ip:bregman} be satisfied, let $q^\dag \in \mathcal{B}_{\rho}(q_0)$ be a solution
 to \eqref{OEq} and let
$F$ be continuous and satisfy Assumption \ref{FseqclosedS}, Assumption \ref{tangconeS} with $c_{tc}$ sufficiently small.
Let, further, $\tau>0$ be chosen sufficiently large  such that
\beq\label{taucondS} 
c_{\cS}\left(c_{\cS}c_{tc}^2 + \frac{1+c_{\cS} c_{tc}^2}{\tau^2}\right) < \tilde{\underline\theta}
\quad \textrm{and}\quad 
0<\frac{c_{\cS}\tilde{\overline\theta} + c_{\cS}^2c_{tc}^2}{1-c_{\cS}^2 c_{tc}^2} < 1
\eeq
and let
\be{noiseS}
  \cS(g,g^\delta)\le \delta^2\,.
\ee
Finally, let for the discretization error with respect to the quantities of interest 
\eqref{eq_IRGNM_I1234kGeneral}, \eqref{eq_IRGNM_I1234kdiscreteGeneral}
estimates \eqref{intcond} hold, 
where $\eta_1^k$, $\eta_2^k$, $\eta_3^k$, $\eta_4^k$ are selected such that
\beq\label{etacond1S} 
\eta_1^{k}+c_{\cS}^2 c_{tc}^2 \eta_3^{k}\leq
\left({\tilde{\underline\theta}} - 
c_{\cS} \left(c_{\cS}c_{tc}^2 + \frac{1+c_{\cS} c_{tc}^2}{\tau^2}\right)\right) I_{3,h}^{k}
\eeq
as well as 
\eqref{etacond2}, the first part of \eqref{etacond3} and 
\be{etacond3S}
  (1+c_3) \frac {c_{\cS} \tilde{\ol \theta} + c_{\cS}^2c_{tc}^2}{1-c_{\cS}^2c_{tc}^2} \le c_2 <1
\ee
hold for some constants $c_1, c_2, c_3
>0$, and a sequence $r^k\to0$ as $k\to\infty$, where \eqref{etacond3S} is possible due to the right inequality in \eqref{taucondS}.

Then with $\beta_k$ and $h=h_k$ fulfilling \eqref{inexNewton} and $k_*$ selected according to \eqref{stop} there holds
\begin{itemize}
\item[(i)] 
For any solution $q^\dagger \in {\cal B}_{\rho}(q_0)$ of \eqref{OEq}
\beq\label{qhkstqdagS}	
   \cR(\qhk) \le \cR(q^\dagger) \quad \forall k<k_*\,,
\eeq
\item[(ii)] $k_*$ is finite,
\item[(iii)] $\qoldkstern = \qhkminuseinsstern$ converges (weakly) subsequentially to a solution of \eqref{OEq} as $\delta\to0$
in the sense that it has a $\tau_Q$ convergent subsequence and  each $\tau_Q$ 
convergent subsequence converges to a solution of \eqref{OEq}.
If the solution $q^\dagger$ to \eqref{OEq} is unique, then $\qoldkstern$ converges 
with respect to $\tau_Q$ to $q^\dagger$ as $\delta\to0$.
\end{itemize}
\end{thm}

\begin{proof}


The proof basically follows the lines of the proof of Theorem \ref{conv}, where we have to replace the specific 
fitting and regularization terms by $\cS$ and $\cR$:
\begin{itemize}
\item[(i):]
For all $k<k_*$ and any solution $q^\dag$ of \eqref{OEq} we have by \eqref{intcond} and
minimality of $\qk$
\begin{equation}\label{eq_minS}
I_{1,h}^{k} \leq I_1^{k} + \eta_1^{k} \leq \cS (F'(\qoldk)(q^\dag-\qoldk) +F(\qoldk),g^\delta) + {\frac 1 {\beta_k}}
\cR(q^\dag) + \eta_1^{k}\,.
\end{equation}
In here, according to \eqref{noiseS},  \eqref{stop} and Assumption \ref{tangconeS},
as well as the inequaltity $(a+b)^2 \le 2a^2 + 2b^2$ for arbitrary $a,b\in \R$ we can estimate as follows
\begin{align}
\cS(F'(\qoldk)(q^\dag-\qoldk) + F(\qoldk),g^\delta) 
&\leq c_{\cS}\left( \cS(g, F'(\qoldk)(q^\dag-\qoldk) + F(\qoldk)) +\delta^2\right)
\nonumber\\
&\leq c_{\cS}\left(c_{tc}^2 \cS(g,F(\qoldk)) +\delta^2\right) \nonumber\\
&\leq c_{\cS}\left(c_{tc}^2 (c_{\cS}(\cS(g^\delta,F(\qoldk))+\delta^2) +\delta^2\right) \nonumber\\
&\leq c_{\cS}^2 c_{tc}^2 I_3^k +  c_{\cS}(1+c_{\cS}c_{tc}^2)\delta^2 \nonumber\\
&\leq c_{\cS}^2 c_{tc}^2(I_{3,h}^{k} + \eta_3^{k})+ c_{\cS}(1+c_{\cS}c_{tc}^2)\frac{I_{3,h}^{k}}{\tau^2}\nonumber\\
&\leq c_{\cS} \left(c_{\cS} c_{tc}^2  + \frac {1+c_{\cS}c_{tc}^2}{\tau^2}\right) I_{3,h}^{k}  + c_{\cS}^2 c_{tc}^2 \eta_3^k
 \label{eq_IRGNM_estTaylorS}\,.
\end{align}
On the other hand, from \eqref{inexNewton} 
and the fact that $I_{1,h}^k = I_{2,h}^k + \frac 1 {\beta_k} \cR(\qhk)$ (cf. \eqref{I1I2})
there follows that 
\begin{equation}\label{eq_IRGNM_min1S}
I_{1,h}^{k}= I_{2,h}^{k} + {\frac 1 {\beta_k}} \cR(\qhk)
\geq {\tilde{\underline\theta}} I_{3,h}^{k}  + {\frac 1 {\beta_k}} \cR(\qhk)\,,
\end{equation}
which together with the previous inequality and \eqref{etacond1S} gives
\begin{align*}
\tilde{\underline\theta} I_{3,h}^{k} +  \frac 1 {\beta_k} \cR(\qhk) 
&\leq I_1^k + \eta_1^k \\
&\leq \cS(F'(\qoldk)(q^\dag-\qoldk) + F(\qoldk),g^\delta) + \frac 1 {\beta_k} \cR(q^\dagger) + \eta_1^k\\
& \leq c_{\cS} \left(c_{\cS} c_{tc}^2  + \frac {1+c_{\cS}c_{tc}^2}{\tau^2}\right) I_{3,h}^{k}  
+ c_{\cS}^2 c_{tc}^2 \eta_3^k
+ \frac 1 {\beta_k} \cR(q^\dagger) + \eta_1^k\\
& \le  {\tilde{\underline\theta}} I_{3,h}^{k} + \frac 1 {\beta_k} \cR(q^\dagger)\,,
\end{align*}
which implies \eqref{qhkstqdagS}.
\item[(ii):] Furthermore, for all  $k< k_*$ we have by the triangle inequality as well as Assumption
\ref{tangconeS} and \eqref{inexNewton}
\begin{align}
I_4^{k} & = \cS(F(\qk),g^\delta)\nonumber\\
& \leq c_{\cS} \left( \cS(F'(\qoldk)(\qk - \qoldk) +F(\qoldk),g^\delta)
+ \cS(F'(\qoldk)(\qk - \qoldk) + F(\qoldk),F(\qk))\right)\nonumber\\
& \leq c_{\cS}\left(I_2^{k} + c_{tc}^2 \cS(F(\qk),F(\qoldk))\right)\nonumber\\
& \leq c_{\cS} \left( I_{2,h}^{k} +\eta_2^k\right) 
+ c_{\cS}^2 c_{tc}^2 \left(\cS(F(\qk),g^\delta) + \cS(F(\qoldk),g^\delta)\right)\nonumber\\
& = c_{\cS} \left( \tilde{\overline\theta} I_{3,h}^{k} +\eta_2^k\right) 
+ c_{\cS}^2 c_{tc}^2 \left(I_4^k + I_3^k\right)\,,
\label{eq_IRGNM_estI4_1S}
\end{align}
which implies
\[
  I_4^k \le \frac 1 {1-c_{\cS}^2c_{tc}^2}\left(c_{\cS} {\tilde{\overline\theta}} I_{3,h}^{k} + c_{\cS} \eta_2^k + c_{\cS}^2c_{tc}^2 I_3^k \right)\,.
\]
With \eqref{intcond} and \eqref{etacond3} we can further deduce
\begin{align*}
I_{4,h}^{k} &\le \frac 1 {1-c_{\cS}^2c_{tc}^2}\left((c_{\cS} {\tilde{\overline\theta}} + c_{\cS}^2c_{tc}^2)I_{3,h}^{k} + c_{\cS} \eta_2^k + c_{\cS}^2 c_{tc}^2 \eta_3^k \right) + \eta_4^k\\
&\le \frac{c_{\cS} {\tilde{\overline\theta}} + c_{\cS}^2 c_{tc}^2}{1-c_{\cS}^2c_{tc}^2} (1+c_3)I_{4,h}^{k-1}  
+ \frac 1 {1-c_{\cS}^2c_{tc}^2} \left(r^k + c_{\cS} \eta_2^k + c_{\cS}^2c_{tc}^2 \eta_3^k \right) + \eta_4^k\\
&\le c_2 I_{4,h}^{k-1}  
+ \frac 1 {1-c_{\cS}^2c_{tc}^2} \left(r^k + c_{\cS} \eta_2^k + c_{\cS}^2c_{tc}^2 \eta_3^k \right) + \eta_4^k\,.
\end{align*}
With the notation 
\be{eq_IRGNM_ai}
  a^i\coloneqq   \frac 1 {1-c_{\cS}^2c_{tc}^2} \left(r^i + c_{\cS} \eta_2^i + c_{\cS}^2c_{tc}^2 \eta_3^i \right) + \eta_4^i \quad\forall i\in \{1,2,\dots,k\}
\ee
there follows recursively 
\be{eq_IRGNM_estI4_2}
  I_{4,h}^k \le c_2^{k}I_{4,h}^0+\sum_{j=0}^{k-1}c_2^j a^{k-j}\,.
\ee

Note that by the second part of \eqref{etacond2}, the second part of \eqref{etacond3} and the fact that $r^k\to 0$ as $k\to\infty$ (by definition 
of $r^k$) , we have $c_2^{k}I_{4,h}^0+\sum_{j=0}^{k-1}c_2^j a^{k-j} \to 0$ as $k\to\infty$. So, if the discrepancy principle never got active (i.e., $k_*=\infty$), the
sequence $(I_{4,h}^k)_{k\in\mathbb{N}}$ and therewith by assumption \eqref{etacond2} also
$(I_{3,h}^k)_{k\in\mathbb{N}}$ would be bounded by a sequence tending to zero as $k\to\infty$, which implies
that $I_{3,h}^k$ would fall below $\tau^2\delta^2$ for $k$ sufficiently large, thus yielding a contradiction. Hence
the stopping index $k_* < \infty$ is well-defined and finite.

\item[(iii):] With \eqref{noiseS}, \eqref{intcond}, \eqref{etacond2} and definition of $k_*$, we have 
\begin{align}
  \cS(F(\qoldkstern),g) &\le c_{\cS}\left( \cS(F(\qoldkstern),g^\delta) + \delta^2\right) \le 
  c_{\cS} \left( I_{3,h}^{k_*} + \eta_3^k + \delta^2 \right) \le c_{\cS}\left((1+c_1)I_{3,h}^{k_*} +\delta^2\right)\nonumber\\
  & \le c_{\cS}\left((1+c_1)\tau^2 +1\right) \delta^2 \to 0
  \label{eq_IRGNM_FconvS}
\end{align}
as $\delta \to 0$.
By (i), (ii) and \eqref{ip:m} in assumption \ref{as:ip:bregman}  
$\qoldkstern = \qhkminuseinsstern$ 
has a $\tau_Q$ convergent subsequence 
$(\qoldksterndeltal)_{l\in\mathbb{N}}$ and due to Assumption \ref{FseqclosedS} and \eqref{eq_IRGNM_FconvS}, 
the limit of every $\tau_Q$ convergent subsequence is a solution to $F(q)=g$.
\end{itemize}
\end{proof}

It is readily checked that (like in the case of $\cR$, $\cS$ being defined by squared Hilbert space norms as
considered in Theorem 1 of \cite{KKV10})  
any approximation $\tilde{q}$  of a  solution $q^\dag$ of $F(q)=g$ with $\|g-g^\delta\|\leq \delta$ such that 
$$    
\cR(\tilde{q})\leq \cR(q^\dag)
\mbox{ and }
\cS(F(\tilde{q}),g^\delta)\leq \hat{\tau}^2 \delta^2
$$
with $\hat{\tau}$ independent of $\delta$, as well as the variational inequality \eqref{source2} holds,
satisfies the rate estimate
$$ D_\cR^\xi(\tilde{q},q^\dag)
\leq \frac{\bar{C}^2\delta^2}{\Theta^{-1} 
\left(\tfrac{\bar{C}}{c}\delta\right)}
= c^2 f^2(\Theta^{-1}(\tfrac{\bar{C}}{c}\delta))
$$
with $\bar{C}^2=c_{\cS}(\hat{\tau}^2+1)$.

Hence we directly obtain from \eqref{qhkstqdagS} and the definition of $k_*$ according to \eqref{stop} 
the following rates result.

\begin{thm} \label{ratesS}
Let the conditions of Theorem \ref{convS} and additionally
the variational inequality \eqref{source2} for some function $f$ with \eqref{fphipsi} be fullfiled.

Then there exists a $\bar{\delta}>0$ and a constant $\bar{C}>0$ independent of $\delta$ such that
for all $\delta\in (0,\bar{\delta}]$
\begin{equation}\label{rates0S}
D_\cR^\xi(\qhkstern,q^\dag)
\leq \frac{\bar{C}^2\delta^2}{\Theta^{-1} 
\left(\tfrac{\bar{C}}{c}\delta\right)}
= c^2 f^2(\Theta^{-1}(\tfrac{\bar{C}}{c}\delta))\,,
\end{equation}
where $\Theta(\lambda)\coloneqq f(\lambda)\sqrt{\lambda}$. 
\end{thm}

\begin{proof}
  The proof can be done analogously to the proof of Theorem 1 in \cite{KKV10} with
  the same replacements as in the proof of Theorem \ref{convS}.
\end{proof}

\section{Conclusions and Remarks}\label{sec_conclusions}
In this paper we consider the iteratively regularized Gauss-Newton method and its adaptive discretization by means of goal oriented error estimators. 
Our aim is to recover convergence as in the continuous setting for discretized hence approximate computations. The key result is that control of a small 
number (four) real valued quantities per Newton step suffices to guarantee convergence and convergence rates. While we have studied the problem in a reduced 
form here, using the parameter-to-solution map, the related paper \cite{KKV13}{(part II of this paper)} develops and studies all-at-once formulations. 
Numerical tests are provided in {part II of this paper} \cite{KKV13}.

\section{Acknowledgments}\label{sec_acknowledgments}
The authors wish to thank Boris Vexler for fruitful discussions, for his support, and for his valuable suggestions after a careful reading of this manuscript.

Moreover, we gratefully acknowlegde financial
support by the German Science Foundation (DFG)  within the grant KA 1778/5-1 and VE 368/2-1 ``Adaptive Discretization
Methods for the Regularization of Inverse Problems''.

\end{document}